\documentclass[a4paper,11pt,twoside,leqno]{article}
\usepackage{amsmath,amssymb,amsfonts,amsthm,graphicx,fancyhdr,tikz}
\usepackage[utf8]{inputenc}
\usepackage[T1]{fontenc}
\usepackage{rotating,authblk}
\usepackage{thmtools}
\usepackage{url}

\newtheorem{teo}{Theorem}[section]
\newtheorem{lem}[teo]{Lemma}
\newtheorem{prop}[teo]{Proposition}
\newtheorem{cor}[teo]{Corollary}
\newtheorem{dfn}[teo]{Definition}

\newtheorem{assu}[teo]{Assumption}

\newtheorem{teoalph}{Theorem}

\declaretheoremstyle[
  spaceabove=\topsep, spacebelow=\topsep,
  headfont=\bf,
  notefont=\mdseries, notebraces={(}{)},
  bodyfont=\rmfamily,
  postheadspace=1em,
  qed=$\Diamond$
]{drem}
\declaretheorem[style=drem, name=Remark, numberlike=teo]{rmk}

\newcommand{\eg}[0]{\emph{e.g.} }
\newcommand{\ie}[0]{\emph{i.e.} }

\newcommand{\srl}[1]{\overline{#1}}

\newcommand{\xvc}[1]{\overrightarrow{#1}}

\newcommand{\eps}[0]{\varepsilon}

\newcommand{\rr}[0]{\ensuremath{\mathbb{R}}}
\newcommand{\zz}[0]{\ensuremath{\mathbb{Z}}}

\numberwithin{equation}{section}

\renewcommand{\thefootnote}{\fnsymbol{footnote}}

\begin{document}

\title{Balancing non-rectangular tables}

\author{Antoine Gournay\footnote{The author gratefully acknowledges partial support from the ERC AdG grant 101097307 and the ERC-StG 277728 ``GeomAnGroup'' grant.}}

\maketitle

\renewcommand{\thefootnote}{\arabic{footnote}}

\centerline{\textsc{Abstract}}

\begin{center}
\parbox{10cm}{{ \small
Balancing square and rectangular tables by rotation has been an interesting way to illustrate the intermediate value theorem.
This article studies the same question for tables whose feet form a cyclic quadrilateral, following a conjecture of A.~Martin. In the idealised model behind this conjecture, the balancing act works for every such table and is a nice application of the ergodic theorem. For the actual rigid table on a Lipschitz terrain, the situation is more subtle. If the terrain is flattened enough, the table can be stabilised unless an idealised wobbling function vanishes identically, \ie unless the terrain is degenerate for the table. Even if this should happen, the tables whose feet are vertices of a regular polygon and which satisfy a non-degeneracy condition of their own can still be stabilised on every terrain with a small enough Lipschitz constant (\eg a halved regular hexagon as long as the slope is less than $3.7$\textdegree{}). On the other hand, there is an isosceles trapezoid and another table, which are shown to wobble at every angle on suitable terrains however small the Lipschitz constant is. These counterexamples are fragile: for every table, if the idealised wobbling function vanishes identically for every position of the rotation axis in a small disc, then, near the feet, the terrain is a plane or a paraboloid of revolution with a vertical axis; so in practice a slight move of the axis makes the terrain non-degenerate. The results are, with minor exceptions, formalised in Lean.

\hspace*{.1ex} 
}}
\end{center}

\section{Introduction}\label{s-intro}

If you have tried to set up a picnic table or used a stepladder, then you are certainly aware of the wobbly table problem.
How can you make this object stop wobbling (without using extra tools)?
Well, if you have a square or even rectangular table, then you are in luck. It turns out that rotating the table around some axis will suffice to eventually stabilise it. Unfortunately, many offices are now full of tables which have a trapezoidal shape. What now?

The trouble is that the arguments for square or rectangular tables hinge crucially on the symmetry of the object. This article will hopefully convince you that you should have no fear of buying tables which are less symmetric, at least if they are cut out of a regular polygon.

Here is a more detailed description of this problem, which was originally made public in Gardner's columns \cite{Ga1} and \cite{Ga2}.
You have some table (or other object) with 4 legs and you find yourself on some terrain. Find conditions on the legs and quite unrestrictive conditions on the terrain, so that turning the table around some axis can make it stop wobbling. Note that the table need not end up level: it will just stop wobbling.

The original arguments of \cite{Ga1} and \cite{Ga2} cover square tables on a continuous terrain, but are presented in an abstract setting which overlooks the rigidity of the problem.
A more realistic set-up is covered by Baritompa, L\"owen, Polster \& Ross \cite{BLPR} and by Martin \cite{Mar}, whose first proof, for square tables, goes back to 1995 and was presented at the IH\'ES in 1998 \cite[\S{}1]{Mar}: there, the terrain needs to be Lipschitz (with a Lipschitz constant bounded above). One of the very first works related to this problem is due to Dyson \cite{Dys}; the reader is referred to \cite{BLPR} for an extensive historical overview.

The subject matter is to discuss an extension to cyclic quadrilaterals (a conjecture raised by Martin \cite[\S{}4]{Mar}). The rough narrative of the proof is presented in \S{}\ref{s-ideal}: in an idealised model, where the feet are put down vertically (basically ignoring the tilting of the table), the argument works for every cyclic quadrilateral. This is Martin's heuristic \cite[\S{}4]{Mar}, which motivates his conjecture. However, much like the first presentations of this problem in \cite{Ga1} and \cite{Ga2}, this narrative overlooks some critical technicalities, which are discussed in \S{}\ref{slook} and dealt with in \S{}\ref{s-transv} and \S{}\ref{s-degen}. The main results are stated in \S{}\ref{ss-intro-results} below.

\subsection{Set-up}\label{ss-intro-setup}

\textbf{The table.} The feet of the table are four points $a_A$, $a_B$, $a_C$, $a_D$ on a circle of radius $r$ centred at $0$ (the table is \emph{cyclic}), which form a convex quadrilateral $ABCD$, in this cyclic order. In complex notation, $a_E = r e^{i\theta_E}$, with $\theta_A = 0$. Let $\tau, \mu \in \,]0,1[$ be such that $\tau a_A + (1-\tau) a_C = \mu a_B + (1-\mu) a_D$ is the intersection point of the diagonals, let $d_{\min}$ be the length of the shortest diagonal, $\varphi \in \,]0, \frac{\pi}{2}]$ the (non-obtuse) angle between the diagonals, and $K = 1/\sin(\varphi/2)$.

\textbf{The terrain} is the graph of a function $g : \rr^2 \to \rr$ which is $L$-Lipschitz, \ie $|g(x) - g(y)| \leq L|x-y|$ for all $x, y \in \rr^2$.

\textbf{Wobbling.} The table is rotated by an angle $\theta$ around the vertical axis through $0$. If the terrain is not too steep, the table can then be put down in a unique way such that the feet $A$ and $C$ are on the ground and the feet $B$ and $D$ are at the same (signed) height $w_{AC}(\theta)$ above the ground: this is the \emph{equal hovering position} of \cite{BLPR}, and $w_{AC}$ is the \emph{wobbling function}. The table stands on its four feet at the angle $\theta$ if and only if $w_{AC}(\theta) = 0$, and we say that the table \emph{can be stabilised} (by rotating it on the terrain $g$) if $w_{AC}$ vanishes somewhere. The precise definitions are given in \S{}\ref{s-transv}.

If the feet were independently put down vertically, the wobbling would be measured by the \emph{idealised wobbling function}
\[
F_0(\theta) = \tau g(R_\theta a_A) + (1-\tau) g(R_\theta a_C) - \mu g(R_\theta a_B) - (1-\mu) g(R_\theta a_D),
\]
where $R_\theta$ is the rotation of angle $\theta$; this ignores the fact that the shadow of a tilted table is shorter than the table. As $F_0$ is continuous and has mean $0$, it vanishes somewhere: this is the heuristic argument of Martin \cite[\S{}4]{Mar}, where $F_0$ already appears, and, in this idealised model, it works for every cyclic table (Theorem \ref{leteo} and Remark \ref{rmk-ivt}(b)).

\textbf{Degenerate terrains and degenerate tables.} If $\hat\gamma(k)$ are the Fourier coefficients of $\varphi \mapsto g(re^{i\varphi})$ (the terrain on the circle), then $F_0(\theta) = -\sum_k m(k)\, \hat\gamma(k)\, e^{ik\theta}$, where
\[
m(k) = \mu e^{ik\theta_B} + (1-\mu) e^{ik\theta_D} - \tau - (1-\tau) e^{ik\theta_C}
\]
(Remark \ref{rmk-fourier}). The integers $k$ with $m(k) = 0$ are the \emph{degenerate modes} of the table; $0$ and $\pm 1$ are always degenerate modes. Two notions of degeneracy come out of this, and they should not be confused:\\
$-$ \; the terrain is \emph{degenerate for the table} if $F_0 \equiv 0$, \ie if the Fourier coefficients of the terrain on the circle are supported in the degenerate modes of the table. This is a property of the pair formed by the table and the terrain: every table has degenerate terrains, for instance the planes.\\
$-$ \; the table is \emph{rational} if its feet are vertices of a regular $n$-gon ($n \geq 4$); then all the $k \equiv 0, \pm 1 \bmod n$ are degenerate modes, and, for the smallest such $n$, the rational table is called \emph{non-degenerate} if it has no other degenerate modes, and \emph{degenerate} otherwise. This is a property of the table alone.

\subsection{Main results}\label{ss-intro-results}

The first result says that the idealised model is correct up to an error of third order in the slope of the terrain. This makes precise the alternative announced by Martin \cite[\S{}4]{Mar}: either the table can be stabilised, or it has approximate equilibrium positions (at the zeros of $F_0$), with an error of third order.

\begin{teoalph}[Propositions \ref{prop-hover} and \ref{prop-error}, Corollary \ref{cor-small}]\label{thmA}
Let the table be cyclic and assume that
\[
LK \leq \frac{1}{2}
\qquad \text{and} \qquad
\frac{4}{\sqrt{3}} \cdot \frac{r}{d_{\min}} \cdot (LK)^2 < 1.
\]
Then the equal hovering position is well defined and depends continuously on $\theta$, and
\[
|w_{AC}(\theta) - F_0(\theta)| \leq 4(2-\sqrt{3})\, r K^2 L^3 \qquad \text{for every } \theta.
\]
Consequently, if $g_1$ is a Lipschitz function whose idealised wobbling function $F_0^{(1)}$ is not identically $0$, then the table can be stabilised by rotating it on the terrain $\eps g_1$, for every $\eps > 0$ small enough.
\end{teoalph}

\textbf{Flat terrains.} A terrain is said to be \textbf{flat enough} (for a given table) if its Lipschitz constant is smaller than some explicit constant which only depends on the geometry of the table, as in the assumptions of Theorem \ref{thmA}. The last part of Theorem \ref{thmA} is weaker: it concerns \textbf{flattened terrains} $\eps g_1$, where the terrain $g_1$ is fixed and $\eps > 0$ is small enough, and how small $\eps$ has to be depends on $g_1$ (through the size of $F_0^{(1)}$), not only on the table.

So the only remaining case is that of a terrain which is degenerate for the table, where $F_0 \equiv 0$ (or, more precisely, where $F_0$ is small compared to $L^3$), \ie where the Fourier coefficients of the terrain on the circle are supported in the degenerate modes of the table. This is not exotic, but it has probability zero for a ``random'' terrain (Remark \ref{rmk-fourier}). The next result is the main positive one: rational tables which are non-degenerate (a property of the table alone, see \S{}\ref{ss-intro-setup}) can be stabilised on all flat enough terrains, including the terrains which are degenerate for them.

\begin{teoalph}[Theorem \ref{teo-orbit} and Corollary \ref{cor-hexagon-L}]\label{thmB}
\begin{enumerate}\renewcommand{\theenumi}{\bfseries (\alph{enumi})}
 \item If the table is rational and non-degenerate, then it can be stabilised by rotating it on every terrain whose Lipschitz constant is small enough. The bound on the Lipschitz constant is explicit and only depends on the table.
 \item In particular, the \emph{hexagonal table}, obtained by cutting a regular hexagon along two opposite vertices (a symmetric trapezoid with three equal sides, with feet at the angles $0$, $\frac{\pi}{3}$, $\frac{2\pi}{3}$ and $\pi$), can be stabilised by rotating it on every terrain with Lipschitz constant $L \leq 0.065$ (slopes up to about $3.7$\textdegree{}).
\end{enumerate}
\end{teoalph}

Part (a) does not require the diagonals to have the same length: for instance, the rational table with feet at the angles $0$, $\frac{\pi}{3}$, $\frac{2\pi}{3}$ and $\frac{4\pi}{3}$ is non-degenerate, and its diagonals have lengths $\sqrt{3}\, r$ and $2r$. A perturbative version of (b), for terrains $\eps g_1$ with $\eps$ small, follows from the third order expansion of the wobbling function (Corollary \ref{cor-hexagon}); Theorem \ref{thmB} needs no expansion. The hexagonal table is the example of a cyclic table which is not a rectangle given by Martin \cite[\S{}4]{Mar} (such tables are found in conference rooms at CERN).

Rectangles are not covered by Theorem \ref{thmB}: they are not rational in general, and rational rectangles are degenerate tables (all odd modes are degenerate). They can however be stabilised: Baritompa, L\"owen, Polster \& Ross \cite[Theorem 2]{BLPR} show that a rectangular table can be balanced by turning it on every terrain with Lipschitz constant at most $1/\sqrt{2}$ (slopes up to about $35.26$\textdegree{}), whatever the shape of the rectangle (for square tables, see also Martin \cite{Mar}). Theorem \ref{teo-rect} is a weaker version of their result in the present set-up.

Rectangles are also special because their diagonals have the same length, and it is suggested in \cite[\S{}4]{Mar} that the method of \cite{BLPR} works for isosceles trapezoids for this reason. However, equal diagonals are not sufficient: some form of rationality is necessary, as there are tables which wobble at every angle, on suitable terrains with arbitrarily small Lipschitz constants.

\begin{teoalph}[Theorem \ref{teo-cex} and Proposition \ref{prop-cex-uneq}]\label{thmC}
\begin{enumerate}\renewcommand{\theenumi}{\bfseries (\alph{enumi})}
 \item Let the table be the isosceles trapezoid with $r = 1$ and feet at the angles $\pm\alpha$ and $\pm\beta$, where $\cos\alpha = \frac{3}{5}$ and $\cos\beta = \frac{\sqrt{23}-3}{10}$ (its diagonals have the same length). For every $0 < \eps \leq 4 \cdot 10^{-3}$, there is an explicit terrain $g_\eps$, polynomial on the unit disc and with Lipschitz constant at most $10\eps + 169\eps^3$, on which
 \[
 w_{AC}(\theta) \geq 0.5\, \eps^3 > 0 \qquad \text{for every } \theta.
 \]
 \item The same holds, with Lipschitz constants at most $85\eps + 331\eps^3$ and $w_{AC} \geq 0.3\,\eps^3$ (for $0 < \eps \leq 10^{-3}$), for a table with diagonals of different lengths (with $r = 1$ and feet at the zeros of $2\cos(2\theta) - 2\cos(\theta + \frac{\pi}{6}) + \frac{1}{2}$).
\end{enumerate}
\end{teoalph}

These are \emph{table--terrain duos}: the table and the terrain have to match. They show that the second case of Martin's alternative does occur: the wobbling is then of third order, but it does not vanish. The terrain is built from a degenerate mode $k \geq 2$ of the table (here $k = 4$, resp.\ $k = 5$), to which the idealised wobbling function is blind, and, at third order, the construction requires a non-rational table with such a degenerate mode (Remark \ref{rmk-cex-comments}(a)). On a flattened terrain $\eps g_1$ which is not degenerate for them (the generic case), the same tables can be stabilised (Theorem \ref{thmA}). Moreover, these duos are fragile, because a terrain which is degenerate for the table does not stay degenerate when the axis of rotation moves. For $v \in \rr^2$, let $F_0^v$ be the idealised wobbling function of the table rotating around the vertical axis through $v$.

\begin{teoalph}[Proposition \ref{prop-axis}]\label{thmD}
Let $U \subset \rr^2$ be an open disc, let $\Omega = \{v + x : v \in U,\ |x| = r\}$ be the region swept by the feet when the axis passes through a point of $U$, and let $g : \Omega \to \rr$ be continuous. Then $F_0^v \equiv 0$ for every $v \in U$ if and only if $g(x) = a + b \cdot x + c|x|^2$ on $\Omega$, for some $a, c \in \rr$ and $b \in \rr^2$.
\end{teoalph}

So, unless the terrain is a plane or a paraboloid of revolution with a vertical axis near the feet, any small move of the axis makes it non-degenerate for the table. For the duos of Theorem \ref{thmC}, the wobbling is of third order in the slope of the terrain. For instance, if the table of Theorem \ref{thmC}(a) has a radius of $1$\,m, it wobbles by at least about $90$\,nm at every angle at the end of the range of the theorem ($\eps = 4 \cdot 10^{-3}$, slopes up to about $2$\textdegree{}); numerically, this minimal wobble is at most about $0.3$\,mm, even for much steeper terrains (slopes up to about $37$\textdegree{}). In both cases, a move of the axis by about $0.15\,\eps^2$ times the radius, \ie by about $3\,\mu$m, resp.\ $1$\,mm, already destroys the wobbling at every angle (Remark \ref{rmk-cex-comments}(f)).

\subsection{Which hypotheses are needed?}\label{ss-intro-hyp}

The following hypotheses were originally considered:\\
$-$ \; the quadrilateral formed by the ends of legs is cyclic; \\
$-$ \; its diagonals are of equal length; \\
$-$ \; the angle (from the centre of the circumscribed circle) which supports a diagonal is rational; \\
$-$ \; the angle which brings one diagonal on the other is also rational.\\
(As in \cite{Mar} or \cite{BLPR}, the terrain needs to be Lipschitz.) Theorems \ref{thmA}--\ref{thmD} show which of them are really needed:\\
$-$ \; cyclicity is necessary for the stabilisation by rotation around a fixed axis (\S{}\ref{s-ideal}), and it is the only hypothesis needed in the idealised model (Theorem \ref{leteo}) and on flattened terrains which are not degenerate for the table (Theorem \ref{thmA});\\
$-$ \; equal diagonals are neither necessary (Theorem \ref{thmB}(a) applies to rational tables with diagonals of different lengths) nor sufficient (Theorem \ref{thmC}(a));\\
$-$ \; the two rationality hypotheses are the relevant ones: for an isosceles trapezoid, they amount to the table being rational; rational non-degenerate tables can be stabilised on every flat enough terrain (Theorem \ref{thmB}(a)), whereas some non-rational tables, with diagonals of equal or of different lengths, wobble at every angle on suitable terrains (Theorem \ref{thmC});\\
$-$ \; the bound on the Lipschitz constant of the terrain is needed for the equal hovering position to be well defined (Proposition \ref{prop-hover}); the bounds of \S{}\ref{s-transv} and \S{}\ref{s-degen} are explicit, but far from optimal.

In particular, the extension of the result on rectangles to cyclic quadrilaterals conjectured by Martin \cite[\S{}4]{Mar} is false, even for isosceles trapezoids, in the form ``every cyclic table can be stabilised by rotating it around a fixed vertical axis on every terrain with a small enough Lipschitz constant'' (Theorem \ref{thmC}). It is false only barely, though: the counterexamples need a terrain which is degenerate for the table, an event of probability zero for a random terrain; the table and the terrain have to be built for each other; the wobbling is of third order in the slope; and it disappears as soon as the axis of rotation moves slightly (Theorem \ref{thmD}). Martin states his conjecture for equilibrium positions in general, which the table may reach by any motion (as in his proof for square tables); in this form, the conjecture remains open, and Theorem \ref{thmD} rather supports it.

Theorem \ref{thmC}(a) also answers, for this trapezoid, a question of Baritompa, L\"owen, Polster \& Ross \cite[\S{}4]{BLPR}: they do not know whether any cyclic table other than a rectangle can always be \emph{balanced locally}, \ie placed with its four feet on the ground and the centre of its circle on the vertical axis, in any position (tilted, upside down or vertical), on every continuous terrain. The trapezoid of Theorem \ref{thmC}(a) cannot (Remark \ref{rmk-cex-comments}(g)).

\subsection{Organisation and formalisation}

Section \ref{s-ideal} presents the argument in the idealised model, and \S{}\ref{slook} discusses what it requires from the touchdowns (conditions (H1)--(H3)) and why the naive ways to obtain them fail. Section \ref{s-transv} makes the equal hovering position precise and proves Theorems \ref{thmA} and \ref{thmD}. Section \ref{s-degen} deals with degenerate terrains: rectangles, the third order term of the wobbling function, the degenerate modes, rational tables, the duos of Theorem \ref{thmC}, and the orbit argument behind Theorem \ref{thmB}. The introduction ends with some open questions (\S{}\ref{ss-open}).

All the statements of \S{}\ref{s-transv} and \S{}\ref{s-degen}, as well as Theorem \ref{leteo} in the idealised model of Remark \ref{rmk-ivt}(b), have been formalised in the proof assistant Lean, including the computer-assisted Theorem \ref{teo-cex} and Proposition \ref{prop-cex-uneq}, and the mathematical statements made in the remarks of \S{}\ref{s-transv} and \S{}\ref{s-degen}; only their numerical observations, the heuristic computation of Remark \ref{rmk-fourier} and the list of the degenerate rational tables with $n \leq 12$ in Remark \ref{rmk-orbit-num}(b) (a numerical check) are not. The Lean code and the Python scripts used in this article are in the public repository \url{https://github.com/gournay/Wobbly-tables} (``the repository'' below); see its directory \texttt{lean/} and the file \texttt{lean/README.md}. The Lean proofs of these last two results (and of Remark \ref{rmk-cex-comments}(g), which relies on the first one) are numerically weaker than the statements given here: they give $w_{AC} \geq \eps^3$ for $0 < \eps \leq 10^{-3}$ (instead of $\eps \leq 4 \cdot 10^{-3}$) in Theorem \ref{teo-cex}, and $w_{AC} \geq 0.3\, \eps^3$ for $0 < \eps \leq 5 \cdot 10^{-4}$ (instead of $\eps \leq 10^{-3}$) in Proposition \ref{prop-cex-uneq}.

\subsection{Open questions}\label{ss-open}

Here are some questions which remain open.
\begin{itemize}
 \item Martin's conjecture \cite[\S{}4]{Mar} in its original form: for every cyclic table and every smooth and flat enough terrain, is there an equilibrium position, \ie a position (reached by any motion of the table, not necessarily by a rotation around a fixed axis) in which the four feet are on the ground? Theorem \ref{thmD} and Remark \ref{rmk-cex-comments}(f) suggest that the duos of Theorem \ref{thmC} are not counterexamples to this form.
 \item The question of Baritompa, L\"owen, Polster \& Ross \cite[\S{}4]{BLPR}: can some cyclic table other than a rectangle, for instance the hexagonal table, always be balanced locally on continuous terrains? The trapezoid of Theorem \ref{thmC}(a) cannot (Remark \ref{rmk-cex-comments}(g)).
 \item The degenerate rational trapezoids, for instance the one with feet at $0, \frac{\pi}{4}, \frac{3\pi}{4}, \pi$, which satisfies all the hypotheses listed in \S{}\ref{ss-intro-hyp} (the other ones with $n \leq 12$ are listed in Remark \ref{rmk-orbit-num}(b)): can they be stabilised on every flat enough terrain? Theorem \ref{thmB}(a) does not apply to them, and neither does the construction behind Theorem \ref{thmC} (Remark \ref{rmk-cex-comments}(a)).
 \item The non-rational tables without degenerate modes $k \geq 2$ (the generic case): can they be stabilised on every flat enough terrain? The third order term cannot produce a counterexample for them (Remark \ref{rmk-cex-comments}(a)).
 \item Are all the degenerate rational tables isosceles trapezoids (or rectangles)? This holds for $n \leq 12$ (Remark \ref{rmk-orbit-num}(b)).
 \item Better constants in Theorem \ref{thmB}(a). The admissible Lipschitz constant necessarily tends to $0$ as $n \to \infty$ (Remark \ref{rmk-orbit-num}(c)); this is forced by Theorem \ref{thmC}.
\end{itemize}

\subsection*{Acknowledgements}

The Lean code was written and cross-checked with the help of the AI systems Claude (Anthropic) and Aristotle (Harmonic), and the Python scripts of the repository were written with the help of Claude.

\section{The idealised set-up}\label{s-ideal}

\subsection{Preliminaries}

Let us start with simple results, reductions, assumptions and notations.
 First off, it seems the very least to demand that your table be stable if the terrain is perfectly flat. Hence the ends of the legs must be on the same plane $P$. The ends of the legs will be denoted by $A$, $B$, $C$ and $D$.

 The axis around which you rotate should be perpendicular to the plane $P$ (you don't want to turn your table upside down!).

 Also, the legs should all run along the same circle (\emph{i.e.} the quadrilateral $ABCD$ is cyclic); otherwise, stabilisation by rotation is not possible (the terrain could be high along the circle described by one leg and low along the circle described by another leg). This necessary condition is pointed out by Martin \cite[\S{}4]{Mar} (see also \cite[\S{}4]{BLPR}). $Z$ will denote the centre of this circle.

 This article stays in a fairly idealised set-up: physical problems, such as ``the terrain is going through the [legs of the] table'' or ``the table will turn over since its centre of mass is ill-placed'', will be ignored. The thickness of the legs is neglected as well (each foot touches the ground at a single point), except for one consequence: in practice, the feet are not points, so the axis of rotation is not perfectly determined and can be moved a little, which only helps (Theorem \ref{thmD} and Remark \ref{rmk-cex-comments}(f)). Martin \cite[\S{}1]{Mar} reports that, on the irregular terrace of the CERN cafeteria, he always succeeded in finding an equilibrium position, although the feet are thick and the ground is sometimes discontinuous.

 The whole problem boils down to the four feet of the table. The actual shape of the table is not really important.

 The table starts in the air. There is some basic position corresponding to the angle $0^\circ$. This means that, in this basic position, the leg $A$ makes an angle of $0^\circ$ as seen from the centre $Z$. The angles $\theta_B$, $\theta_C$ and $\theta_D$ are the angles made by the other legs in this basic position. (For convenience, $\theta_A = 0$.)

 As you turn your table, all the angles of the legs are changed by the same amount. (The table is rigid after all.)

\subsection{Putting the table down}

The next step is to look at possible touchdowns for your table. 
As you rotate the table, consider the following way to put it down. 
First, you could put the legs $A$ and $C$ on the ground (and completely ignore where $B$ and $D$ are). 
If $B$ and $D$ are above the terrain, then the table wobbles; if not, then this was not completely legal, but it is still useful to think about it (as some sort of negative wobbling).
In this idealised set-up, the legs only move vertically during the touchdown. This is not realistic (when the table wobbles, the feet $B$ and $D$ describe arcs of circles around the line $AC$), but it is in the spirit of Gardner's columns \cite{Ga1}, \cite{Ga2}. The touchdown is discussed much more carefully by Baritompa, L\"owen, Polster \& Ross \cite{BLPR}, and it is made precise in this article in \S{}\ref{s-transv} (the equal hovering position).

Next, you could do the same thing with $B$ and $D$ touching down first.
Note that these touchdowns should happen on the same curve (some distorted circle); otherwise, there is no hope to conclude. 
This will not be the case in general, see \cite{BLPR} and \S{}\ref{stouchdown} for a correct description of this process, but in this section we stay in this idealised set-up.

Let $X$ be the intersection point of the diagonals $AC$ and $BD$. If we make a touchdown with $AC$, the coordinates of $X$ do not change when the table wobbles. 
Likewise with $BD$. 
This gives us two height functions: $h_{AC}(\theta)$ is the height ($z$-coordinate) of $X$ after an $AC$-touchdown, where $\theta$ is an angle of rotation (of the table). 
And likewise for $h_{BD}$.

\begin{lem}\label{lelem}
If the table wobbles, then $h_{BD}(\theta) \neq h_{AC}(\theta)$.
\end{lem}
\begin{proof}
If the table wobbles, then one of the pairs $BD$ or $AC$ can be pushed further down. As a consequence, the two heights of $X$ differ.
\end{proof}

The coordinates of $X$ are a convex combination of the coordinates of the legs. More precisely, it is a small exercise with vectors that, if $\tau = \frac{ \srl{CX}}{ \srl{CA}}\in ]0,1[$ and $\mu = \frac{ \srl{DX}}{ \srl{DB}}\in ]0,1[$, then
\[
\xvc{OX} = \tau \xvc{OA} + (1-\tau) \xvc{OC} 
= \mu \xvc{OB} + (1-\mu) \xvc{OD}
\]
In particular, this holds for the height coordinate:
\[
h_{AC}(\theta) = \tau h_A(\theta) + (1-\tau) h_C(\theta)
\qquad 
\text{and}
\qquad
h_{BD}(\theta) = \mu h_B(\theta) + (1-\mu) h_D(\theta)
\]
where $h_E(\theta)$ gives the height of (the end of) leg $E$ after a touchdown after a rotation [of the table] by the angle $\theta$ and $E \in \{A,B,C,D\}$. 
Since every leg can be brought to the position of any other leg by a rotation, the functions $h_A$, $h_B$, $h_C$ and $h_D$ are all identical up to a translation (see \S{}\ref{slook} and \S{}\ref{s-transv} for details).

\subsection{Table turning}

\begin{teo}\label{leteo}
Assume that the functions $h_E$ are continuous. Then there is an angle $\theta$ at which the table does not wobble.
\end{teo}
\begin{proof}
Since $h_E$ is measurable, let $\int h_E = H$ (the average height).
By assumption $h_E$ is continuous, hence so are $h_{BD}$ and $h_{AC}$. 
Let $h_\Delta = h_{BD} - h_{AC}$.

Assume that the table cannot be stabilised. Then, by Lemma \ref{lelem}, $h_\Delta$ never vanishes. Since $h_\Delta$ is continuous, there is an $\eps>0$ so that, 
\[
\text{either }
\qquad 
\forall \theta,\; h_{\Delta}(\theta)  > \eps
\qquad 
\text{or }
\qquad
\forall \theta,\; h_{\Delta}(\theta)  < -\eps
\]
Without loss of generality, we may assume the first holds.
Let $\theta_0$ be an irrational angle. 
Then 
\[
\forall N, \quad \frac{1}{N} \sum_{i=1}^N h_\Delta(i \cdot \theta_0) > \eps
\]
On the other hand
\[
\frac{1}{N} \sum_{i=1}^N h_\Delta(i \cdot \theta_0)
=\frac{1}{N}  \sum_{i=1}^N \bigg(
 \mu   h_B(i \cdot \theta_0) +
 (1-\mu)   h_D(i \cdot \theta_0)  -
 \tau  h_A(i \cdot \theta_0) -
 (1-\tau)  h_C(i \cdot \theta_0)
 \bigg)
\]
By the equidistribution theorem or the ergodic theorem (see, among many possibilities, Arnold \& Avez \cite[Appendix 1]{AA}, Hardy \& Wright \cite[Section 23.10]{HW}, Navas \cite[Exercise 2.2.12]{Nav} and Zorzi \cite{Z})
\[
 \lim_{N \to \infty} \frac{1}{N} \sum_{i=1}^N h_\Delta(i \cdot \theta_0)
 = 
 \int h_\Delta = 
 \mu \int  h_B +
 (1-\mu) \int  h_D -
  \tau \int h_A -
 (1-\tau) \int h_C 
\]
But the right-hand side is just 
\[
\mu H + (1-\mu)H - \tau H - (1-\tau)H = H - H = 0.
\]
Hence
\[
 \lim_{N \to \infty} \frac{1}{N} \sum_{i=1}^N h_\Delta(i \cdot \theta_0)
  = 0 \geq \eps
\]
a contradiction.
\end{proof}

\begin{rmk}\label{rmk-ivt}
(a) The ergodic theorem can be replaced by the intermediate value theorem: $h_\Delta$ is a continuous function on the circle with $\int h_\Delta = 0$ (this only uses (H3) of \S{}\ref{stouchdown} and the invariance of the Lebesgue measure under rotations), hence it cannot have a constant sign. This is the argument of Martin \cite[\S{}4]{Mar}, who uses the mean value theorem.

(b) In the most idealised model, the table is ``put down'' by moving the ends of the legs vertically, \ie the horizontal positions of the feet at angle $\theta$ are exactly the points of the circle at angles $\theta + \theta_E$ (this ignores the fact that the shadow of a tilted table is shorter than the table). If the terrain is the graph of $g$, then $h_E(\theta) = g(\theta + \theta_E)$, four feet are coplanar if and only if $h_{AC}(\theta) = h_{BD}(\theta)$, and the above proof works for \emph{any} cyclic quadrilateral: this observation is the origin of Martin's conjecture \cite[\S{}4]{Mar}. The assumptions of \S{}\ref{slook} are only needed to pass from this idealised model to the actual problem, see \S{}\ref{s-transv}.

(c) In the idealised model of (b), the ergodic part of the argument does not need any continuity: if $g$ is merely integrable on the circle, then so is $h_\Delta$, with $\int h_\Delta = 0$, and Birkhoff's ergodic theorem for the rotation by an irrational angle $\theta_0$ shows that $\frac{1}{N} \sum_{i=1}^N h_\Delta(\theta + i \theta_0) \to 0$ for almost every $\theta$. So $h_\Delta$ cannot stay above some $\eps > 0$ (or below $-\eps$): it is impossible that, at every angle, the same diagonal can be pushed down by a definite amount. This is of course not really helpful: without continuity, $h_\Delta$ may change sign without vanishing (think of a terrain with a step), so that no angle needs to stabilise the table.
\end{rmk}

\section{A closer look}\label{slook}

\subsection{Assumptions}\label{stouchdown}

A tool from Baritompa, L\"owen, Polster \& Ross \cite[\S{}3]{BLPR} (see also \S{}\ref{s-transv}) is to look only at the touchdown with respect to $AC$ and then consider how far the vertices $B$ and $D$ are in the air. This is done so that both vertices $B$ and $D$ are equally far from the ground and is called the ``equal hovering position''.

Let us sketch how to apply Theorem \ref{leteo} in this context. 
Consider the average (over all angles) of $h_{AC}$, $\int h_{AC}$.
Since the average of $h_{BD}$ is equal to that of $h_{AC}$, the equal hovering position is on average 0.
But the equal hovering position is continuously defined, hence if the average is $0$, it should be 0 somewhere.
When the equal hovering position is 0, then the table does not wobble. 

However, there is a hidden hypothesis in this argument. Namely, if the two diagonals do not have the same length, then it could be that $h_{BD}$ takes on different values than $h_{AC}$: variations in the steepness of the terrain could make the pair of vertices $B$ and $D$ linger different times in different regions. This suggests requiring that
\begin{equation}\label{ass-len}
\text{the diagonals of } ABCD \text{ have the same length.}
\end{equation}
As it turns out, \eqref{ass-len} is neither necessary (Theorem \ref{teo-orbit}) nor sufficient (Theorem \ref{teo-cex}).

As mentioned before, \cite{BLPR} (and \S{}\ref{s-transv}) contains a detailed description of how to realise the touchdown of two opposite vertices ($AC$ or $BD$).
The proof of Theorem \ref{leteo} requires the height functions $h_A$, $h_B$, $h_C$ and $h_D$:
\begin{enumerate}\renewcommand{\theenumi}{\bfseries (H\arabic{enumi})}
 \item to be well-defined;
 \item to be continuous;
 \item to differ only by a translation, \emph{i.e.} $h_E(\theta) = h_A(\theta + \theta_E)$.
\end{enumerate}
The difficulty is that, while letting the table down, you have to tilt it, so that the feet do not land on the circle used in the idealised model of Remark \ref{rmk-ivt}(b); moreover, as pointed out by Martin \cite{Mar}, there are many ways to put a table down. In \S{}\ref{s-transv}, (H1) and (H2) are established for the equal hovering position when the Lipschitz constant of the terrain is small enough (Proposition \ref{prop-hover}), whereas (H3) only holds up to an error of third order in the Lipschitz constant (Proposition \ref{prop-error}).

The touchdowns also play an implicit role in Lemma \ref{lelem}: for a given angle $\theta$, the touchdowns with respect to $AC$ and to $BD$ have to be done in one go. So you cannot let the table down and see where $AC$ ends up, and then let the table down again and see where $BD$ ends up: you need to let the table down, see which of the two pairs touches the ground first, and then go on with the other one. This requires a transversality condition, namely that the height of $X$ decreases while the table is let down. It holds when the terrain is Lipschitz with a small enough constant, see \cite{BLPR} and \cite{Mar}: in \S{}\ref{s-transv}, the equal hovering positions with respect to $AC$ and to $BD$ turn out to have the same tilt, so that the feet land at the same points in both (Proposition \ref{prop-common}).

\begin{rmk}[The naive chain of touchdowns]\label{rmk-naive-chain}
A natural way to obtain (H3) is to make the height functions $h_E$ depend only on the angle (and on $\theta_E$), by defining the touchdowns along a chain. The touchdown $h_{BD}(0)$ has to be defined simultaneously with the touchdown $h_{AC}(0)$. But then $h_{AC}(\theta_B)$, $h_{AC}(\theta_C)$ and $h_{AC}(\theta_D)$ are also defined (the leg $A$ has to land where the legs $B$, $C$ and $D$ did). Repeating this process, all the angles $k_1 \theta_B + k_2 \theta_C + k_3 \theta_D$ (where $k_1,k_2,k_3 \in \zz_{\geq 0}$) have to be dealt with at once. In order to avoid an infinite number of choices, this leads to the following two assumptions.
\end{rmk}

\begin{assu}\label{ass-rot2}
The rotation bringing one diagonal on the other diagonal is rational.
\end{assu}
\begin{assu}\label{ass-rot}
The rotation bringing one end of a diagonal on the other end is rational (\ie the angle supporting the diagonals from the centre of the circumscribed circle is rational).
\end{assu}

Note that \eqref{ass-len} and Assumption \ref{ass-rot} are both automatically verified in the case of a rectangle (diagonals of equal length is a characteristic feature of rectangles among parallelograms; the rotation is by $180$\textdegree{} since the centre of the circle lies on the diagonals), but Assumption \ref{ass-rot2} might not hold. For an isosceles trapezoid, Assumptions \ref{ass-rot2} and \ref{ass-rot} together amount to the table being rational, \ie its feet are vertices of a regular polygon. The chain of Remark \ref{rmk-naive-chain} does not work as such (Remark \ref{rmk-chain}), but the finite orbits of rotations which it suggests are the basis of Theorem \ref{teo-orbit}, which only needs the table to be rational (and non-degenerate), whether or not its diagonals have the same length.

\begin{rmk}\label{rmk-density}
One might be able to get rid of the rationality assumptions. 
A standard way to do so would be to consider a sequence of tables with rational angles which tend to the desired (irrational) table. 
Since everything is happening over a compact part of $\rr^2$, the stabilising positions of the rational tables will converge to (at least) one position.
This position should stabilise the (irrational) table.
However, this strategy fails: the results of \S{}\ref{s-degen} imply that, for rational tables approximating a fixed irrational table, the denominator $n$ tends to infinity and the Lipschitz constant allowed by Theorem \ref{teo-orbit} tends to $0$ (Remark \ref{rmk-orbit-num}(c)), so there is no common terrain on which all the approximating tables are known to balance. Even worse, it cannot work in general: some irrational tables wobble at every angle on suitable terrains (Theorem \ref{teo-cex} and Proposition \ref{prop-cex-uneq}).
\end{rmk}
 
\subsection{Invariant measures}\label{sinvmeas}

The idealised proof relies on a very specific invariant measure: the uniform measure on the circle, which is invariant under rotations. When (H3) fails, one could try to replace the multiples $i \cdot \theta_0$ of an irrational angle in the proof by a sequence of angles $\theta_i$ whose averages tend to another, well-chosen, measure. This is not needed for the results of this article (and, by Theorem \ref{teo-cex}, no argument can work for all tables), but the question has a nice answer, which we record here.

More concretely, let $h_B(\theta) = h_A(\theta + f_B(\theta))$, $h_C(\theta) = h_A(\theta + f_C(\theta))$ and $h_D(\theta) = h_A(\theta + f_D(\theta))$. The functions $x \mapsto x + f_E(x)$ (where $E \in \{B,C,D\}$) generate a monoid (of circle maps). If the action of this monoid is amenable, then there is an invariant measure $\mu$ (see Greenleaf \cite{Gr} for general background on amenability and Navas \cite[\S{}2.3 and \S{}3.2]{Nav} for more specific details in the case of groups acting on the circle). Since the convex hull of Dirac masses is weak$^*$ dense in the space of means, one can then choose the sequence of angles $\theta_i$ so that $\tfrac{1}{N}\sum_{i=1}^N \delta_{\theta_i}$ tends (weak$^*$) to $\mu$, and the argument of Theorem \ref{leteo} goes through with $\mu$ instead of the uniform measure.

A simple way of satisfying the amenability assumption is to assume that the monoid is generated by some rational rotations (\eg coming from Assumptions \ref{ass-rot} and \ref{ass-rot2}) and by one more map (which need not be a rotation). The monoid is then a finite extension of an amenable one (a single map generates a monoid isomorphic to $\mathbb{N}$), hence it is amenable as well. Yves de Cornulier \cite{Cor} kindly pointed out that \cite{Nav} essentially shows that there are no other examples where this monoid is amenable (besides the given example where two maps are rational).

\section{Equal hovering}\label{s-transv}

The aim of this section is to make the ``equal hovering position'' of \S{}\ref{stouchdown} precise, to show that it is well-defined and continuous in the rotation angle (conditions (H1) and (H2) of \S{}\ref{stouchdown}), and to compare the resulting wobbling function with the idealised one of Remark \ref{rmk-ivt}(b).

\subsection{Tilted tables and their shadows}

\textbf{The terrain} is the graph of a function $g: \rr^2 \to \rr$ which is $L$-Lipschitz, \ie $|g(x) - g(y)| \leq L |x-y|$ for all $x,y \in \rr^2$. (If $g$ is $C^1$, this is the same as $|\nabla g| \leq L$.)

\textbf{The table.} In the frame of the table, the feet are at $a_E = r(\cos \theta_E, \sin \theta_E) \in \rr^2$ for $E \in \{A,B,C,D\}$, where $r$ is the radius of the circumscribed circle (centred at $Z = 0$), $\theta_A = 0$, and $ABCD$ is a convex quadrilateral (in this cyclic order). Let $d_1 = |a_A - a_C|$ and $d_2 = |a_B - a_D|$ be the lengths of the diagonals, $d_{\min} = \min(d_1,d_2)$, let $\varphi \in \,]0,\pi/2]$ be the (non-obtuse) angle between the lines $AC$ and $BD$ and let
\[
K = \frac{1}{\sin(\varphi/2)} \in [\sqrt{2}, +\infty[.
\]
As in \S{}\ref{s-ideal}, $\tau, \mu \in \,]0,1[$ are such that $a_X := \tau a_A + (1-\tau) a_C = \mu a_B + (1-\mu) a_D$ is the intersection of the diagonals.

\textbf{Positions of the table.} Let $\mathbb{B} = \{ q \in \rr^2 \mid |q| < 1\}$. For $q \in \mathbb{B}$, let $S_q \in SO(3)$ be the rotation around the horizontal axis orthogonal to $q$, by the angle $\arcsin |q|$, in the direction which lifts $q$ upwards ($S_0 = \mathrm{Id}$). Let $R_\theta$ be the rotation of angle $\theta$ in $\rr^2$. The position $(\theta,q,c) \in \rr/2\pi\zz \times \mathbb{B} \times \rr$ of the table is the one where the foot $E$ is at
\[
S_q (R_\theta a_E,0) + (0,0,c) \in \rr^3.
\]
In words: $\theta$ is the rotation angle, $q$ is the tilt ($|q|$ is the sine of the angle between the table and the horizontal plane, and $q/|q|$ is the direction of steepest ascent) and $c$ is the height of the centre $Z$, which stays on the vertical axis through $0$.

Every rotation $T \in SO(3)$ which does not turn the table upside down nor makes it vertical can be written in a unique way as $T = S_q \circ R_\theta$ (where $R_\theta$ is seen as a rotation around the vertical axis): $S_q$ is the ``shortest'' rotation which brings the horizontal plane onto the plane of the table. Hence the only arbitrary choices in this parametrisation are what ``rotation by $\theta$'' means for a tilted table and the fact that $Z$ stays on the vertical axis. These choices are natural, but other (continuous) choices would work as well in what follows. This is one of the non-uniqueness issues mentioned in \S{}\ref{stouchdown} (and in \cite{Mar}).

\begin{lem}\label{lem-coord}
Let $u = R_\theta a_E$ (so $|u| = r$). In the position $(\theta,q,c)$, the foot $E$ is at $(x_E, z_E) \in \rr^2 \times \rr$ where
\[
x_E = u - \Phi_u(q),
\qquad
\Phi_u(q) = \frac{(u \cdot q)\, q}{1 + \sqrt{1 - |q|^2}},
\qquad
z_E = c + u \cdot q.
\]
In particular, the height $z_E$ is an affine function of $u$ and, since $a_X$ is the same convex combination of $a_A$, $a_C$ and of $a_B$, $a_D$,
\begin{equation}\label{eq-plane}
\tau z_A + (1-\tau) z_C = \mu z_B + (1-\mu) z_D.
\end{equation}
\end{lem}
\begin{proof}
Write $\hat q = q/|q|$ (if $q = 0$ there is nothing to prove) and $u = (u \cdot \hat q) \hat q + u^\perp$. $S_q$ fixes $(u^\perp,0)$ and sends $(\hat q,0)$ to $(\sqrt{1-|q|^2}\, \hat q, |q|)$. Hence the horizontal part is $u - (1- \sqrt{1-|q|^2})(u \cdot \hat q) \hat q$ and the vertical part is $|q|(u \cdot \hat q) = u \cdot q$. Conclude with $1 - \sqrt{1-s} = s/(1+\sqrt{1-s})$.
\end{proof}

Note that $x_E$ is the image of $u$ by the linear map $\pi \circ S_q$ (where $\pi$ is the projection on the horizontal plane), which does not increase lengths. In particular $|x_A - x_C| \leq d_1$ and $|x_B - x_D| \leq d_2$: the shadow of a tilted table is smaller than the table. The difference $\Phi_u(q)$ between the actual position of the foot and the idealised one is quadratic in the tilt:

\begin{lem}\label{lem-phi}
Let $|u| = r$, $0 < t \leq 1/2$ and $q,q' \in \rr^2$ with $|q|, |q'| \leq t$. Then
\[
|\Phi_u(q)| \leq \frac{r t^2}{1+\sqrt{1-t^2}} \leq 2(2-\sqrt{3})\, r t^2
\qquad \text{and} \qquad
|\Phi_u(q) - \Phi_u(q')| \leq \frac{2}{\sqrt{3}}\, r t \, |q-q'|.
\]
\end{lem}
\begin{proof}
Let $f(s) = 1/(1+\sqrt{1-s})$, so that $\Phi_u(q) = f(|q|^2) (u \cdot q) q$. Both $f$ and $f'(s) = \big(2\sqrt{1-s}(1+\sqrt{1-s})^2\big)^{-1}$ are positive and increasing on $[0,1[$. The first bound follows from $|u \cdot q| \leq r|q|$. For the second,
\[
D\Phi_u(q)[v] = 2 f'(|q|^2) (q\cdot v)(u \cdot q) q + f(|q|^2)\big( (u \cdot v) q + (u \cdot q) v \big),
\]
so that $\|D\Phi_u(q)\| \leq r \psi(t)$ where $\psi(t) = 2 t f(t^2) + 2 t^3 f'(t^2)$. Since $\psi(t)/t$ is increasing, $\psi(t) \leq 2 \psi(1/2)\, t = \frac{2}{\sqrt{3}} t$ for $t \leq 1/2$. The claim follows from the mean value inequality on the (convex) disc of radius $t$.
\end{proof}

\begin{lem}\label{lem-diag}
Let $w_1,w_2 \in \rr^2$ be unit vectors such that the (non-obtuse) angle between the lines $\rr w_1$ and $\rr w_2$ is $\varphi > 0$. If $q \in \rr^2$ satisfies $|w_1 \cdot q| \leq s$ and $|w_2 \cdot q| \leq s$, then $|q| \leq K s$ where $K = 1/\sin(\varphi/2)$.
\end{lem}
\begin{proof}
Let $s_i = w_i \cdot q$. Inverting the Gram matrix of $(w_1,w_2)$ gives $|q|^2 = (s_1^2 + s_2^2 - 2 s_1 s_2 \,w_1 \cdot w_2)/\sin^2 \varphi$. This is maximal at the vertices of the square $|s_i| \leq s$, where it is at most $s^2(2+2\cos\varphi)/\sin^2\varphi = 2s^2/(1-\cos\varphi) = s^2/\sin^2(\varphi/2)$.
\end{proof}

\subsection{Hovering and wobbling}

The \emph{equal hovering position with respect to $AC$} at angle $\theta$ is a position $(\theta,q,c)$ in which the feet $A$ and $C$ are on the ground and the feet $B$ and $D$ are at the same (signed) height above the ground:
\begin{equation}\label{eq-hover}
z_A = g(x_A), \qquad z_C = g(x_C), \qquad z_B - g(x_B) = z_D - g(x_D).
\end{equation}

\begin{prop}\label{prop-hover}
Assume that
\begin{equation}\label{eq-smallL}
LK \leq \frac{1}{2}
\qquad \text{and} \qquad
\lambda := \frac{4}{\sqrt{3}} \cdot \frac{r}{d_{\min}} \cdot (LK)^2 < 1.
\end{equation}
Then for every $\theta$, there is exactly one equal hovering position $(\theta, q_{AC}(\theta), c_{AC}(\theta))$ with respect to $AC$, and $|q_{AC}(\theta)| \leq LK$. The functions $q_{AC}$ and $c_{AC}$ are continuous. If moreover $g$ is $C^k$ ($1 \leq k \leq \infty$), then $q_{AC}$ and $c_{AC}$ are $C^k$.
\end{prop}
\begin{proof}
Let $u_E = R_\theta a_E$ and write $x_E(q) = u_E - \Phi_{u_E}(q)$. By Lemma \ref{lem-coord}, $z_A - z_C = (u_A - u_C) \cdot q$ and $z_B - z_D = (u_B - u_D) \cdot q$, so that \eqref{eq-hover} is equivalent to
\[
(u_A - u_C) \cdot q = g(x_A(q)) - g(x_C(q)),
\qquad
(u_B - u_D) \cdot q = g(x_B(q)) - g(x_D(q)),
\qquad
c = g(x_A(q)) - u_A \cdot q.
\]
The last equation determines $c$ once $q$ is known. Let $w_1 = (u_A - u_C)/d_1$, $w_2 = (u_B-u_D)/d_2$ (unit vectors making an angle $\varphi$). Since the diagonals are not parallel, for every $q \in \mathbb{B}$ there is a unique $Q_\theta(q) \in \rr^2$ such that
\[
w_1 \cdot Q_\theta(q) = \frac{g(x_A(q)) - g(x_C(q))}{d_1}
\qquad \text{and} \qquad
w_2 \cdot Q_\theta(q) = \frac{g(x_B(q)) - g(x_D(q))}{d_2},
\]
and the first two equations say that $q$ is a fixed point of $Q_\theta$. This linear part is the transversality: the two diagonals give two independent conditions on the tilt.

\emph{$Q_\theta$ maps $\mathbb{B}$ into the closed disc $\bar{\mathbb{B}}_\rho$ of radius $\rho = LK \leq 1/2$.} Indeed $|x_A(q) - x_C(q)| \leq d_1$ (the shadow is smaller than the table), so the right-hand sides are bounded by $L$ in absolute value, and Lemma \ref{lem-diag} gives $|Q_\theta(q)| \leq LK$.

\emph{$Q_\theta$ is a contraction on $\bar{\mathbb{B}}_\rho$.} For $q, q' \in \bar{\mathbb{B}}_\rho$ and each $E$, Lemma \ref{lem-phi} gives $|g(x_E(q)) - g(x_E(q'))| \leq L |\Phi_{u_E}(q) - \Phi_{u_E}(q')| \leq \frac{2}{\sqrt 3} L r \rho |q-q'|$. Hence the right-hand sides above change by at most $\frac{4}{\sqrt{3}} \frac{r}{d_{\min}} L \rho |q-q'|$ and, by Lemma \ref{lem-diag}, $|Q_\theta(q) - Q_\theta(q')| \leq \lambda |q-q'|$.

By Banach's fixed point theorem, $Q_\theta$ has a unique fixed point in $\bar{\mathbb{B}}_\rho$, and any fixed point in $\mathbb{B}$ lies in $\bar{\mathbb{B}}_\rho$ (the image of $Q_\theta$). This shows existence and uniqueness.

$(\theta,q) \mapsto Q_\theta(q)$ is continuous, so the fixed point $q_{AC}(\theta)$ depends continuously on $\theta$: $|q_{AC}(\theta) - q_{AC}(\theta')| \leq (1-\lambda)^{-1} |Q_\theta(q_{AC}(\theta)) - Q_{\theta'}(q_{AC}(\theta))|$. If $g$ is $C^k$, then $(\theta,q) \mapsto Q_\theta(q) - q$ is $C^k$ and its differential in $q$ is invertible (since $\|D_q Q_\theta\| \leq \lambda <1$), so the implicit function theorem shows that $q_{AC}$ is $C^k$. The same then holds for $c_{AC}$.
\end{proof}

\begin{rmk}
(a) Only the Lipschitz constant of $g$ is needed for (H1) and (H2); smoothness of $g$ only improves the regularity of the equal hovering position. The condition \eqref{eq-smallL} is certainly not optimal. For a square table ($K = \sqrt{2}$, $d_{\min} = 2r$) it reads $L \leq 1/(2\sqrt{2})$, \ie slopes of at most about $19.5$\textdegree{}. Compare with the Lipschitz constant $1/\sqrt{2}$ (slopes up to about $35.26$\textdegree{}) for all rectangles in Baritompa, L\"owen, Polster \& Ross \cite[Theorem 2]{BLPR}, which is optimal for their argument, and with slopes up to about $14.5$\textdegree{} (and $35.26$\textdegree{} with a different motion) for square tables in Martin \cite[\S{}3]{Mar}.

(b) The condition \eqref{eq-smallL} is symmetric in the two diagonals, so the same holds for the \emph{equal hovering position with respect to $BD$} (the feet $B$ and $D$ on the ground, $A$ and $C$ at the same height).

(c) \eqref{eq-smallL} does not involve the lengths of the diagonals being equal, nor any rationality assumption.
\end{rmk}

\begin{dfn}
Under \eqref{eq-smallL}, the \emph{wobbling function} $w_{AC}(\theta)$ is the common value of $z_B - g(x_B) = z_D - g(x_D)$ in the equal hovering position with respect to $AC$. The function $w_{BD}(\theta)$ is defined likewise, with the roles of the diagonals exchanged.
\end{dfn}

By Proposition \ref{prop-hover}, $w_{AC}$ and $w_{BD}$ are continuous (and $C^k$ if $g$ is $C^k$). If $w_{AC}(\theta) > 0$, then the table rests on $A$ and $C$ and wobbles; if $w_{AC}(\theta) <0$, then the legs $B$ and $D$ go through the ground (the ``negative wobbling'' of \S{}\ref{s-ideal}). The following statement replaces Lemma \ref{lelem}.

\begin{lem}\label{lem-wobble}
Assume \eqref{eq-smallL}. For every angle $\theta$, the following are equivalent:
\begin{enumerate}\renewcommand{\theenumi}{\bfseries (\alph{enumi})}
 \item there is a position $(\theta,q,c)$ (with $q \in \mathbb{B}$) in which the four feet are on the ground;
 \item $w_{AC}(\theta) = 0$;
 \item $w_{BD}(\theta) = 0$.
\end{enumerate}
\end{lem}
\begin{proof}
A position with the four feet on the ground satisfies \eqref{eq-hover}, hence is the equal hovering position with respect to $AC$ by uniqueness, so $w_{AC}(\theta)=0$. The converse is clear. The same argument works for $BD$.
\end{proof}

\begin{prop}\label{prop-w}
Assume \eqref{eq-smallL} and let $x_E(\theta)$ be the horizontal position of the foot $E$ in the equal hovering position with respect to $AC$. Then
\[
w_{AC}(\theta) = \tau g(x_A(\theta)) + (1-\tau) g(x_C(\theta)) - \mu g(x_B(\theta)) - (1-\mu) g(x_D(\theta)).
\]
\end{prop}
\begin{proof}
Put $z_A = g(x_A)$, $z_C = g(x_C)$, $z_B = g(x_B) + w_{AC}(\theta)$ and $z_D = g(x_D) + w_{AC}(\theta)$ in \eqref{eq-plane}.
\end{proof}

\begin{prop}\label{prop-common}
Assume \eqref{eq-smallL}. For every $\theta$, the equal hovering positions with respect to $AC$ and with respect to $BD$ have the same tilt, $q_{BD}(\theta) = q_{AC}(\theta)$, and $c_{BD}(\theta) = c_{AC}(\theta) - w_{AC}(\theta)$. In particular, the feet are at the same horizontal positions in both, and $w_{BD} = - w_{AC}$.
\end{prop}
\begin{proof}
In the proof of Proposition \ref{prop-hover}, the two equations which determine $q$ say that $z_A - g(x_A) = z_C - g(x_C)$ and $z_B - g(x_B) = z_D - g(x_D)$. They are symmetric in the two diagonals, so the equal hovering position with respect to $BD$ has the same tilt (the unique fixed point of $Q_\theta$). For this tilt, changing $c$ shifts the four residuals $z_E - g(x_E)$ by the same amount. With respect to $AC$, they are $(0, w_{AC}, 0, w_{AC})$ (in the order $A,B,C,D$); subtracting $w_{AC}$ from $c$ gives $(-w_{AC},0,-w_{AC},0)$, which is the equal hovering position with respect to $BD$.
\end{proof}

So, with $h_E(\theta) := g(x_E(\theta))$ (the height of the ground under the foot $E$), $w_{AC} = h_{AC} - h_{BD} = - h_\Delta$ in the notation of \S{}\ref{s-ideal}, and (H1) and (H2) hold. However, (H3) does not hold in general: $x_E(\theta)$ differs from the idealised position $R_\theta a_E$ by $\Phi_{R_\theta a_E}(q_{AC}(\theta))$, which depends on the whole table and not only on $\theta + \theta_E$. By Proposition \ref{prop-common}, the positions $x_E(\theta)$ are the same in the equal hovering position with respect to $BD$: the touchdowns of $AC$ and of $BD$ are indeed done ``in one go'', as required in \S{}\ref{stouchdown}. The error is however of third order in $L$:

\begin{prop}\label{prop-error}
Assume \eqref{eq-smallL} and let
\[
F_0(\theta) = \tau g(R_\theta a_A) + (1-\tau) g(R_\theta a_C) - \mu g(R_\theta a_B) - (1-\mu) g(R_\theta a_D)
\]
be the idealised wobbling function (see Remark \ref{rmk-ivt}(b)); in particular $\int_0^{2\pi} F_0 = 0$. Then, for every $\theta$,
\[
|w_{AC}(\theta) - F_0(\theta)| \leq \frac{2 r K^2 L^3}{1 + \sqrt{1-(LK)^2}} \leq 4(2-\sqrt{3})\, r K^2 L^3.
\]
\end{prop}
\begin{proof}
By Proposition \ref{prop-w}, $w_{AC}(\theta) - F_0(\theta)$ is a combination with coefficients $\tau, 1-\tau, -\mu$ and $-(1-\mu)$ (whose absolute values sum to $2$) of the terms $g(x_E(\theta)) - g(R_\theta a_E)$. Each of these is bounded by $L|\Phi_{R_\theta a_E}(q_{AC}(\theta))| \leq L r (LK)^2 / (1+\sqrt{1-(LK)^2})$, by Lemma \ref{lem-phi} and $|q_{AC}(\theta)| \leq LK$.
\end{proof}

\begin{cor}\label{cor-small}
Let $g_1: \rr^2 \to \rr$ be a Lipschitz function and let
\[
F_0^{(1)}(\theta) = \tau g_1(R_\theta a_A) + (1-\tau) g_1(R_\theta a_C) - \mu g_1(R_\theta a_B) - (1-\mu) g_1(R_\theta a_D)
\]
be its idealised wobbling function (the function $F_0$ of Proposition \ref{prop-error}, with $g$ replaced by $g_1$). Assume that $F_0^{(1)}$ is not identically $0$. Then there is $\eps_0 >0$ such that, for every $0 < \eps \leq \eps_0$, the table can be stabilised by rotating it on the terrain $g = \eps g_1$.
\end{cor}
\begin{proof}
Since $F^{(1)}_0$ is continuous, has integral $0$ and is not identically $0$, \linebreak $m := \min(\max F^{(1)}_0, -\min F^{(1)}_0) >0$. The idealised wobbling function of $\eps g_1$ is $\eps F_0^{(1)}$ and its Lipschitz constant is $\eps L_1$. For $\eps$ small enough, \eqref{eq-smallL} holds and, by Proposition \ref{prop-error}, $|w_{AC} - \eps F_0^{(1)}| \leq 4(2-\sqrt{3}) r K^2 L_1^3 \eps^3 < \eps m$. Hence $w_{AC}$ takes both signs, so it vanishes somewhere, and one concludes with Lemma \ref{lem-wobble}.
\end{proof}

\subsection{Degenerate terrains and a moving axis}\label{ss-remains}

In general, the argument of \S{}\ref{s-ideal} applied to $w_{AC}$ only gives $\int w_{AC} = \int (w_{AC} - F_0)$, which is small ($O(L^3)$) but has no reason to vanish. So the remaining difficulty is the case of a terrain which is degenerate for the table ($F_0 \equiv 0$), or nearly so ($F_0$ small compared to $L^3$): then the wobbling is decided by the correction term of Proposition \ref{prop-error}. This case is studied in \S{}\ref{s-degen}.

\begin{rmk}\label{rmk-fourier}
{\bfseries The case $F_0 \equiv 0$ is not exotic...} Write $g(R_\theta a_A) = \sum_k \hat g_k e^{ik\theta}$ (the Fourier series of $g$ on the circle). Then $F_0(\theta) = -\sum_k m(k) \hat g_k e^{ik\theta}$ where
\[
m(k) = \mu e^{ik\theta_B} + (1-\mu) e^{ik\theta_D} - \tau - (1-\tau) e^{ik\theta_C}.
\]
Always $m(0) = m(\pm 1) = 0$ (the latter is the fact that the diagonals meet at $X$). For a rectangle, $m(k) = 0$ for all odd $k$. Hence, if $g$ restricted to the circle only has Fourier modes $k$ with $m(k)=0$, then $F_0 \equiv 0$. Here is a heuristic (second order) computation, still to be made rigorous. Take $g = \eps g_1 + \beta(|x| - r)$ with $\eps$ small, where the Fourier modes of $g_1$ on the circle satisfy $m(k) = 0$. The shadow of the table is (to second order) moved radially inwards by $\frac{r}{2}(q\cdot e_E)^2$ at the foot $E$ (where $e_E$ is the unit vector towards $E$), so that
\[
w_{AC}(\theta) \approx -\frac{r\beta}{2}\, q^{T} M q,
\qquad \text{where} \qquad
M = \tau e_A e_A^T + (1-\tau) e_C e_C^T - \mu e_B e_B^T - (1-\mu) e_D e_D^T.
\]
Since $\mathrm{tr}\, M = 0$, $M$ is not definite. If $g_1$ has a single Fourier mode $\pm k$ on the circle, then $q$ (seen in the frame of the table) describes an ellipse centred at $0$, and the inequality $|v_1^2 - v_2^2| \geq \big| |v_1|^2 - |v_2|^2 \big|$ (for $v_1,v_2 \in \mathbb{C}$) shows that the quadratic form does not have a constant sign along this ellipse (it vanishes somewhere). However, this only concerns a \emph{constant} radial slope $\beta$; for a radial slope depending on the angle, this heuristic does not extend, see \S{}\ref{ss-wobble} (in particular Theorem \ref{teo-cex}).

{\bfseries ... not exotic, but not expected either.} For a ``random'' terrain, $F_0 \equiv 0$ requires the vanishing of all the Fourier coefficients $\hat g_k$ with $m(k) \neq 0$, an event of probability zero. So if your table wobbles at every angle, the floor was probably designed by a mathematician.
\end{rmk}

In practice, one does not rotate a table around a perfectly fixed axis: the axis moves a little. Moving the vertical axis to the point $v \in \rr^2$ amounts to replacing the terrain $g$ by $g(v + \cdot)$, which changes its Fourier coefficients on the circle. A terrain which is degenerate for the table does not stay degenerate under this operation, unless it is trivial:

\begin{prop}\label{prop-axis}
Let $U \subset \rr^2$ be an open disc, let $\Omega = \{v + x : v \in U,\ |x| = r\}$ be the region swept by the feet when the axis passes through a point of $U$, and let $g : \Omega \to \rr$ be continuous. For $v \in U$, let $F_0^v$ be the idealised wobbling function of the table rotating around the vertical axis through $v$, \ie the function $F_0$ of Proposition \ref{prop-error} for the terrain $g(v + \cdot)$.
\begin{enumerate}\renewcommand{\theenumi}{\bfseries (\alph{enumi})}
 \item $F_0^v \equiv 0$ for every $v \in U$ if and only if $g(x) = a + b \cdot x + c|x|^2$ on $\Omega$, for some $a, c \in \rr$ and $b \in \rr^2$.
 \item For every continuous $g : \rr^2 \to \rr$, $F_0^v \equiv 0$ for every $v \in \rr^2$ if and only if $g(x) = a + b \cdot x + c|x|^2$ on $\rr^2$. If $g$ is moreover Lipschitz on $\rr^2$, then $g$ is affine.
\end{enumerate}
\end{prop}

So a terrain which is degenerate for the table ($F_0 \equiv 0$) becomes non-degenerate after \emph{any} small move of the axis, unless it is a plane or a paraboloid of revolution with a vertical axis near the feet. (If $U$ has centre $v_0$ and radius $s < r$, $\Omega$ is the annulus $r - s < |x - v_0| < r + s$; the terrain inside is never touched by the feet and is arbitrary.) On such a paraboloid, any cyclic table balances at every angle when it rotates around the axis of the paraboloid.

The formula for $F_0$ only involves the values of $g$ at the feet, and the proof of Proposition \ref{prop-hover} only evaluates $g$ at the shadows of the feet, which lie in the closed disc of radius $r$ around the axis. So for the actual table rotating around the axis through $v$, only the Lipschitz constant of $g$ on $\bar B(v, r)$ matters: a function which is $L$-Lipschitz on $\bar B(v,r)$ extends to an $L$-Lipschitz function on $\rr^2$, for instance $x \mapsto \min_{y \in \bar B(v,r)} \big(g(y) + L|x - y|\big)$ (McShane \cite{McS}), without changing anything under the table. Thus a paraboloid $c|x|^2$ is a legitimate terrain for every axis $v$ with $2|c|(|v| + r) \leq L$, and the paraboloids are genuinely exceptional terrains, which stay degenerate for every table after every (moderate) move of the axis; global Lipschitz bounds only matter if the axis is allowed to go to infinity.

\begin{proof}
\emph{Notation and converse.} In complex notation, let
\[
G_k(v) = \frac{1}{2\pi}\int_0^{2\pi} g(v + r e^{i\varphi})\, e^{-ik\varphi}\, d\varphi \qquad (v \in U)
\]
be the Fourier coefficients of $g$ on the circle of radius $r$ centred at $v$, so that $G_{-k} = \overline{G_k}$. By Remark \ref{rmk-fourier}, $F_0^v \equiv 0$ if and only if $G_k(v) = 0$ for all $k$ in the set $N = \{k \mid m(k) \neq 0\}$ of non-degenerate modes. We have $N = -N$ (as $m(-k) = \overline{m(k)}$), $0, \pm 1 \notin N$ (Lemma \ref{lem-modes}(a)), $2 \in N$ (Lemma \ref{lem-modes}(e)), $3 \in N$ unless the table is a rectangle (Lemma \ref{lem-modes}(f)), and for a rectangle $m(4) = e^{4i\theta_B} - 1 \neq 0$ unless the table is a square (Lemma \ref{lem-modes}(b)). Since $|v + re^{i\varphi}|^2 = |v|^2 + r^2 + 2r\,\mathrm{Re}(\bar v e^{i\varphi})$ only has the Fourier modes $0$ and $\pm 1$ in $\varphi$, as do affine functions, the paraboloids satisfy $F_0^v \equiv 0$ for every $v$. For a function $f$ on $\Omega$, we write $G_k[f]$ for its coefficients (so $G_k = G_k[g]$) and $Z(f) = \{k \in \zz \mid G_k[f] \equiv 0 \text{ on } U\}$.

\emph{Proof of (a). Reduction to smooth terrains.} Let $U = B(v_0, R_0)$. For $0 < \delta < R_0$, let $U_\delta = B(v_0, R_0 - \delta)$ and $\Omega_\delta$ the corresponding region. The mollification $g_\delta = g * \rho_\delta$ (with $\rho_\delta$ supported in $B(0,\delta)$) is smooth on $\Omega_\delta$, and its coefficients are $G_k * \rho_\delta$, which vanish on $U_\delta$ for $k \in N$. If $g_\delta$ is a paraboloid on $\Omega_\delta$ for every $\delta$, then, as $g_\delta \to g$ uniformly on compact subsets of $\Omega$ and the paraboloids form a finite-dimensional space, $g$ is a paraboloid on $\Omega$. So we may assume that $g$ is smooth. Let $\partial = \frac{1}{2}(\partial_x - i \partial_y)$ and $\bar\partial = \frac{1}{2}(\partial_x + i\partial_y)$. For a smooth $f$ on $\Omega$, differentiation under the integral gives $\partial G_k[f] = G_k[\partial f]$ and $\bar\partial G_k[f] = G_k[\bar\partial f]$ (the derivatives of $G_k[f]$ act on the centre $v$), so that
\[
Z(f) \subset Z(\partial f) \cap Z(\bar\partial f).
\]
Since $\frac{d}{d\varphi} f(v + re^{i\varphi}) = ir e^{i\varphi}\, \partial f - ir e^{-i\varphi}\, \bar\partial f$, an integration by parts gives, for all $k \in \zz$ and $v \in U$,
\begin{equation}\label{eq-rec}
k\, G_k[f] = r\big(G_{k-1}[\partial f] - G_{k+1}[\bar\partial f]\big) = r\big(\partial G_{k-1}[f] - \bar\partial G_{k+1}[f]\big).
\end{equation}

\emph{Step 1: isolated modes.} If $k \neq 0$ and $k-1, k+1 \in Z(f)$, then $k \in Z(f)$: by \eqref{eq-rec}, as $k - 1 \in Z(\partial f)$ and $k+1 \in Z(\bar\partial f)$.

\emph{Step 2: blocks of modes.} Let $a \leq b$ with $0 \notin [a, b]$. We show by induction on $b - a$ that, for every smooth $f$ on $\Omega$,
\[
a-1,\ b+1,\ b+2 \in Z(f) \quad \Longrightarrow \quad [a, b] \subset Z(f).
\]
If $a = b$, this is Step 1. If $a < b$, \eqref{eq-rec} for $f$ at $k = b+1$ gives $G_b[\partial f] = 0$ on $U$, as $b+1 \in Z(f)$ and $b+2 \in Z(\bar\partial f)$; so $a - 1, b, b+1 \in Z(\partial f)$, and the induction hypothesis for $\partial f$ and the block $[a, b-1]$ gives $b - 1 \in Z(\partial f)$. Then \eqref{eq-rec} for $f$ at $k = b$ gives $b\, G_b[f] = r(G_{b-1}[\partial f] - G_{b+1}[\bar\partial f]) = 0$ on $U$, so $b \in Z(f)$, and the induction hypothesis for $f$ and the block $[a, b-1]$ (with $b, b+1 \in Z(f)$) gives $[a, b-1] \subset Z(f)$.

\emph{Step 3: the modes $|k| \geq 2$.} Assume first that the table is not a square. We have $N \subset Z(g)$, so $2 \in Z(g)$, and $3 \in Z(g)$ (for a rectangle which is not a square, by Step 1, as $2, 4 \in N$). Moreover, $Z(g)$ contains pairs $\{n+2, n+3\}$ with $n$ arbitrarily large: by compactness of the torus $(S^1)^4$ (or Dirichlet's simultaneous approximation theorem) there are $n_j \to \infty$ with $e_E^{n_j} \to 1$ for the four feet $e_E = e^{i\theta_E}$, so that $m(n_j + k) \to m(k)$ for every $k$, and the argument which gave $2, 3 \in Z(g)$ gives $n_j + 2, n_j + 3 \in Z(g)$ for $j$ large. Step 2 with $a = 4$ and $b = n_j + 1$ gives $k \in Z(g)$ for all $k \geq 2$, and $G_{-k} = \overline{G_k}$ gives $k \in Z(g)$ for all $|k| \geq 2$.

\emph{Step 3 for the square.} For the square, $e_E^4$ is the same for the four feet, so that $m(4j+k) = e_A^{4j} m(k)$: as $m(3) = m(4) = 0$ and $m(2) \neq 0$, the non-degenerate modes are the $k \equiv 2 \bmod 4$, and the pairs of consecutive modes used above do not exist. Let $\mathcal{C}$ be the set of smooth functions $f$ on $\Omega$ such that $4j+2 \in Z(f)$ for every $j \geq 0$; it contains $g$, and it is stable under $\partial$ and $\bar\partial$, which commute. Let $f \in \mathcal{C}$ and $m = 4j$ with $j \geq 1$. By \eqref{eq-rec} for $f$ at $k = m$, for $\partial f$ at $k = m-1$ and for $\bar\partial f$ at $k = m+1$ (as $m-2 \in Z(\partial f)$ and $m+2 \in Z(\bar\partial f)$),
\[
(m^2-1)\, G_m[f] = -2r^2\, G_m[\partial\bar\partial f],
\]
so $G_m[f]$ is an eigenfunction of $\partial\bar\partial$, with an eigenvalue which depends on $m$. By \eqref{eq-rec} for $f$ at $k = m+2$, for $\partial f$ at $k = m+1$ and for $\bar\partial f$ at $k = m+3$,
\[
(m+3)\, G_m[\partial^2 f] = -(m+1)\, G_{m+4}[\bar\partial^2 f].
\]
Let $x = G_m[\partial^2 f]$ and $y = G_{m+4}[\bar\partial^2 f]$. The second identity for $\partial\bar\partial f$, together with the first one for $\partial^2 f$ (at $m$) and for $\bar\partial^2 f$ (at $m+4$), gives $(m+1)(m+3)\big((m-1)x + (m+5)y\big) = 0$, while $(m+3)x + (m+1)y = 0$; hence $8(m+2)\,x = 0$, \ie $m \in Z(\partial^2 f)$. Using the first identity twice, $(m^2-1)^2\, G_m[f] = 4r^4\, G_m[\partial^2 (\bar\partial^2 f)] = 0$, so $m \in Z(f)$. So $4j \in Z(f)$ for every $j \geq 1$, and then \eqref{eq-rec} at $k = 4j+3$ and $k = 4j+5$ gives $4j+3, 4j+5 \in Z(f)$ for every $j \geq 0$. Hence $k \in Z(g)$ for all $k \geq 2$, and for all $|k| \geq 2$.

\emph{Step 4: the modes $0, \pm 1$.} By \eqref{eq-rec} for $g$ at $k = 1$ and $k = 2$ (with $2, 3 \in Z(\bar\partial g)$), $G_1 = r\, \partial G_0$ and $1 \in Z(\partial g)$, and then \eqref{eq-rec} for $\partial g$ at $k = 1$ (with $2 \in Z(\bar\partial\partial g)$) gives $0 \in Z(\partial^2 g)$, \ie $\partial^2 G_0 = 0$ on $U$. Now $G_0$ is real. Let $A = \partial G_0$, so that $\partial A = 0$, and let $C = \bar\partial A = \frac14 \Delta G_0$, which is real. Then $\partial C = \bar\partial \partial A = 0$ (symmetry of second derivatives) and $\bar\partial C = \overline{\partial C} = 0$, so $C$ is a constant $c$ on the disc $U$. Then $\partial(A - c\bar v) = 0$ and $\bar\partial(A - c \bar v) = 0$, so $A = \beta + c\bar v$ for some $\beta \in \mathbb{C}$. Finally $G_0 - (\beta v + \bar\beta \bar v + c|v|^2)$ is real and has $\partial = 0$, so it is a constant $\alpha$: $G_0 = \alpha + \beta v + \bar \beta \bar v + c|v|^2$ and $G_1 = r(\beta + c\bar v)$, $G_{-1} = \overline{G_1}$ on $U$. These are exactly the coefficients of the paraboloid $p(x) = \alpha - cr^2 + \beta x + \bar\beta \bar x + c|x|^2$ (in complex notation), whose coefficients $G_k$, $|k| \geq 2$, vanish. So $g$ and $p$ have the same Fourier coefficients on every circle of radius $r$ centred in $U$, and $g = p$ on $\Omega$.

\emph{Proof of (b).} Apply (a) to the discs $B(0, R)$, $R \to \infty$: the paraboloids obtained agree on overlaps, since a paraboloid which vanishes on a non-empty open set is $0$. If $g$ is Lipschitz on $\rr^2$, then $c = 0$.

\emph{A global argument for (b).} Here is a different proof of (b), by spectral synthesis, which does not use Steps 1--4. It works for every table; we write it for the square, for which $N = \{k \equiv 2 \bmod 4\}$. Let $V$ be the space of continuous functions $h : \rr^2 \to \mathbb{C}$ such that $G_k[h](v) = 0$ for every $v$ and every $k$ with $m(k) \neq 0$. If $g$ is real-valued and $F_0^v \equiv 0$ for every $v$, then $g \in V$. The space $V$ is closed for the topology of uniform convergence on compact sets, and it is invariant under translations and under rotations (rotating the terrain by $\alpha$ replaces $G_k(v)$ by $e^{ik\alpha} G_k(e^{-i\alpha} v)$). By the theorem of Brown, Schreiber \& Taylor \cite{BST}, $V$ is the closed linear span of the exponential polynomials $p(x) e^{\zeta \cdot x}$ ($p$ a polynomial, $\zeta \in \mathbb{C}^2$) which it contains; and if $p(x)e^{\zeta\cdot x} \in V$ with $p \neq 0$, then $e^{\zeta \cdot x} \in V$ (apply the operators $h \mapsto h(\cdot + w) - e^{\zeta \cdot w} h$, which preserve $V$ and lower the degree of $p$). For $h = e^{\zeta\cdot x}$, write $\zeta \cdot (\cos\varphi, \sin\varphi) = \alpha e^{i\varphi} + \beta e^{-i\varphi}$ with $\alpha = \frac12(\zeta_1 - i\zeta_2)$, $\beta = \frac12(\zeta_1 + i\zeta_2)$. Then $G_k[h](v) = e^{\zeta\cdot v} c_k$ with
\[
c_k = \sum_{j \geq 0} \frac{r^{k+2j}\alpha^{k+j}\beta^j}{(k+j)!\, j!} \qquad (k \geq 0).
\]
If $\beta = 0 \neq \alpha$, then $c_2 = \frac12 (r\alpha)^2 \neq 0$; similarly (with $c_{-2}$) if $\alpha = 0 \neq \beta$. If $\alpha\beta \neq 0$, then $c_k = \lambda^k J_k(s)$ for some $\lambda \neq 0$, where $J_k$ is the Bessel function and $s^2 = -4r^2\alpha\beta = -r^2\, \zeta\cdot\zeta \neq 0$. As $J_k$ only has real zeros for $k \geq 0$ (see Watson \cite[\S{}15.25]{Wat}), and by Bourget's hypothesis, proved by Siegel, $J_{2}$ and $J_{6}$ have no common non-zero zero \cite[\S{}15.28]{Wat}, $c_{2}$ and $c_{6}$ do not both vanish ($m(2), m(6) \neq 0$ for the square). Hence $\zeta = 0$: the exponential polynomials in $V$ are polynomials.

The polynomials in $V$ are paraboloids, by the following lemma (as $\pm 2 \in N$), and conversely. As the paraboloids form a finite-dimensional, hence closed, space, $V$ consists of the paraboloids and $g(x) = a + b \cdot x + c|x|^2$.

\emph{Lemma.} Let $p : \rr^2 \to \mathbb{C}$ be a polynomial whose coefficients $G_2[p]$ and $G_{-2}[p]$ vanish on $\rr^2$. Then $p = a + b \cdot x + c|x|^2$ (with complex coefficients). Indeed, the space of such polynomials is invariant under translations, hence under derivatives. If $p$ has degree $\leq 2$, write $p = a + b \cdot x + \alpha z^2 + \alpha' \bar z^2 + c|z|^2$; then $G_2[p] = \alpha r^2$ and $G_{-2}[p] = \alpha' r^2$, so $\alpha = \alpha' = 0$. If $p$ had degree $d \geq 3$, some derivative of order $d-3$ of $p$ would be such a polynomial of degree exactly $3$; its first derivatives have degree $\leq 2$, so $\partial_x p = c_1 |x|^2 + \ell_1$ and $\partial_y p = c_2|x|^2 + \ell_2$ with $\ell_1, \ell_2$ affine, and $\partial_y \partial_x p = \partial_x \partial_y p$ gives $2c_1 y + \partial_y \ell_1 = 2 c_2 x + \partial_x \ell_2$, so $c_1 = c_2 = 0$: a contradiction.
\end{proof}

The proof of (a), including Step 3 for the square, is the one formalised in Lean. Here is an alternative to Steps 2 and 4, which shows that the coefficients are polynomials. Let $a \leq b$ be such that $a-2, a-1, b+1, b+2 \in Z(g)$, and let $H_k = G_k$ for $a \leq k \leq b$ and $H_k = 0$ otherwise. Then the $H_k$ satisfy \eqref{eq-rec} (in the form $k H_k = r(\partial H_{k-1} - \bar\partial H_{k+1})$) for every $k$. Writing it as $\partial H_{k-1} = \frac{k}{r} H_k + \bar\partial H_{k+1}$, an induction downwards from $H_{b+1} = 0$ and $\partial H_b = 0$ shows that $\partial^{\,b+1-k} H_k = 0$ for $k \leq b+1$; symmetrically, $\bar\partial^{\,k-a+1} H_k = 0$ for $k \geq a-1$. A smooth function $f$ on a disc with $\partial^p f = \bar\partial^q f = 0$ is a polynomial in $v$ and $\bar v$ (by induction on $q$: $h = \bar\partial^{q-1} f$ is holomorphic with $\partial^p h = 0$, hence a polynomial, and $f - h\, \bar v^{q-1}/(q-1)!$ satisfies the same equations with $q-1$). So the $H_k$ are polynomials, which still satisfy \eqref{eq-rec} on $\rr^2$. Then the finite sum $H(v, \varphi) = \sum_k H_k(v) e^{ik\varphi}$ satisfies $\partial_\varphi H = ir e^{i\varphi}\, \partial H - ir e^{-i\varphi}\, \bar\partial H$, so that $\varphi \mapsto H(x - re^{i\varphi}, \varphi)$ is constant, and $p(x) = H(x - r, 0)$ is a polynomial whose coefficients on the circle of radius $r$ centred at $v$ are the $H_k(v)$, for every $v \in \rr^2$. If $a \geq 3$, then $G_{\pm 2}[p] = 0$, so $p$ is a paraboloid by the lemma, and $H_k = 0$ for $|k| \geq 2$; this gives Step 3 (with $a = 4$). For $a = -1$, $b = 1$, $p$ is the paraboloid of Step 4.

\section{Degenerate terrains}\label{s-degen}

Throughout this section, \eqref{eq-smallL} holds and $(c_A,c_B,c_C,c_D) = (\tau,-\mu,1-\tau,-(1-\mu))$, so that
\[
w_{AC}(\theta) = \sum_E c_E\, g(x_E(\theta)),
\qquad
F_0(\theta) = \sum_E c_E\, g(R_\theta a_E),
\qquad
\sum_E c_E = 0,
\qquad
\sum_E c_E a_E = 0
\]
(Proposition \ref{prop-w}; the last identity says that $X$ is on both diagonals).

\subsection{Rectangles}

For rectangles, all the terrains which satisfy \eqref{eq-smallL}, degenerate or not, can be handled by the argument of Baritompa, L\"owen, Polster \& Ross \cite{BLPR}. It relies on the fact that both diagonals are diameters of the circle, of the same length.

\begin{teo}[Baritompa, L\"owen, Polster \& Ross \cite{BLPR}]\label{teo-rect}
If $ABCD$ is a rectangle and \eqref{eq-smallL} holds, then there is an angle $\theta$ such that the table stands on its four feet.
\end{teo}

This is \cite[Theorem 2]{BLPR} (for the feet of the table) in the present set-up, under a stronger assumption: \cite{BLPR} only assumes that the Lipschitz constant is at most $1/\sqrt{2}$, whatever the shape of the rectangle, and also deals with real tables, whose legs and top must stay above the ground. Their argument carries over, and we only sketch it; the details are in the Lean formalisation (file \texttt{lean/Wobbly/Rectangle.lean} of the repository). Here $X = Z$, and a rod of length $2r$ whose midpoint is on the vertical axis and whose two ends are on the ground is determined by the direction $\varphi$ of its shadow (as $L < 1$); let $H(\varphi)$ be the height of its midpoint. In the equal hovering position with respect to $AC$, resp.\ $BD$, the diagonal $AC$, resp.\ $BD$, is such a rod, so that $c_{AC}(\theta) = H(\varphi_A(\theta))$ and $c_{BD}(\theta) = H(\varphi_B(\theta))$, where $\varphi_A(\theta)$ and $\varphi_B(\theta)$ are the directions of the shadows of $A$ and $B$; and $w_{AC} = c_{AC} - c_{BD}$ by Proposition \ref{prop-common}. The maps $\varphi_A$ and $\varphi_B$ have degree $1$, hence are onto: if $H$ is maximal at $\varphi^*$ and $\varphi_A(\theta_1) = \varphi^* = \varphi_B(\theta_2)$, then $w_{AC}(\theta_1) \geq 0 \geq w_{AC}(\theta_2)$, so $w_{AC}$ vanishes somewhere, and one concludes with Lemma \ref{lem-wobble}.

For an isosceles trapezoid which is not a rectangle, the diagonals do not go through $Z$, and the height of $Z$ is not a function of the position of one rod only (it also depends on how the table turns around the rod). This is where this argument breaks down.

\subsection{The third order term}\label{ss-3ord}

From now on and until the end of \S{}\ref{ss-wobble}, $r = 1$ (by scaling, this is no loss of generality) and the points of the plane are identified with complex numbers, so that $a_E = e_E := e^{i\theta_E}$ and $u \cdot v = \mathrm{Re}(\bar u v)$. Let $\bar{\mathbb{D}}$ be the closed unit disc. The terrain is now of the form
\[
g = \eps g_1 + \eps^3 g_3 \quad \text{on } \bar{\mathbb{D}},
\]
where $g_1$ and $g_3$ are $C^2$ on a neighbourhood of $\bar{\mathbb{D}}$ and $\eps > 0$.

This is a perturbation expansion in the size $\eps$ of the terrain: $g_1$ is a fixed shape, and $\eps g_1$ is a flattened terrain (see \S{}\ref{ss-intro-results}). By Corollary \ref{cor-small}, the table can then be stabilised unless $F_0^{(1)} \equiv 0$, \ie unless $g_1$ is degenerate for the table. In this case, the leading term $\eps F_0^{(1)}$ of the wobbling function vanishes, and its sign is decided by the next order, which is the third one: the correction of Proposition \ref{prop-error} is cubic, because the feet are moved horizontally by $\Phi_u(q) = O(|q|^2) = O(\eps^2)$, which changes the heights by $O(\eps^3)$. Proposition \ref{prop-w3} below computes this third order term. The term $\eps^3 g_3$ is a correction of the same order, which is used to cancel some Fourier modes of the third order term in \S{}\ref{ss-wobble} (in \S{}\ref{ss-rational}, $g_3 = 0$).

There is no term of order $\eps^2$, for two reasons. First, replacing $g$ by $-g$ replaces the equal hovering position $(\theta, q, c)$ by $(\theta, -q, -c)$ (as $\Phi_u(-q) = \Phi_u(q)$), hence $w_{AC}$ by $-w_{AC}$; so, on the terrain $\eps g_1 + \eps^3 g_3$, $w_{AC}$ is an odd function of $\eps$, and only odd powers of $\eps$ appear in Proposition \ref{prop-w3}. Second, a term $\eps^2 g_2$ would only contribute its idealised wobbling function $\eps^2 F_0^{(2)}$ at order $\eps^2$ (Proposition \ref{prop-error}): if $F_0^{(2)} \not\equiv 0$, the table can be stabilised as in Corollary \ref{cor-small}, and if $F_0^{(2)} \equiv 0$, the term $\eps^2 g_2$ only contributes at order $\eps^4$, beyond the order considered here.

Only the values of $g$ on $\bar{\mathbb{D}}$ matter, since the feet stay in $\bar{\mathbb{D}}$ (the shadow of the table is smaller than the table); outside $\bar{\mathbb{D}}$, $g$ is extended by $g(x) = g(x/|x|)$, which does not increase its Lipschitz constant. Let
\[
L_1 = \max_{\bar{\mathbb{D}}} |\nabla g_1|,
\qquad
M_1 = \max_{\bar{\mathbb{D}}} \|D^2 g_1\|,
\qquad
L_3 = \max_{\bar{\mathbb{D}}} |\nabla g_3|,
\]
so that $g$ is $L$-Lipschitz with $L = \eps L_1 + \eps^3 L_3$. Let $u_E = R_\theta a_E$, let $q_0(\theta)$ be the gradient of the (affine) idealised touchdown for $g_1$, \ie the unique vector with
\[
(u_A - u_C)\cdot q_0(\theta) = g_1(u_A) - g_1(u_C)
\qquad \text{and} \qquad
(u_B - u_D)\cdot q_0(\theta) = g_1(u_B) - g_1(u_D),
\]
let $q_3(\theta)$ be the same vector for $g_3$, and let $F_0^{(1)}$ and $F_0^{(3)}$ be the idealised wobbling functions of $g_1$ and $g_3$. As in the proof of Lemma \ref{lem-phi}, $f(s) = 1/(1+\sqrt{1-s})$.

\begin{prop}\label{prop-w3}
Let $P \geq \max_\theta |q_0(\theta)|$ and $P_3 \geq \max_\theta |q_3(\theta)|$, and assume that $L = \eps L_1 + \eps^3 L_3$ satisfies \eqref{eq-smallL}. Put $t_0 = LK$,
\[
t_1 = \min\Big( t_0,\ \eps P + \eps^3 P_3 + \frac{2KL}{d_{\min}}\, f(t_0^2)\, t_0^2 \Big),
\qquad
\Phi_1 = f(t_1^2)\, t_1^2,
\qquad
\delta = \eps^3 P_3 + \frac{2KL}{d_{\min}}\, \Phi_1,
\]
\[
R(\eps) = \eps M_1 \Phi_1^2 + 2 \eps L_1 \Big( \big(f(t_1^2) - \tfrac{1}{2}\big) t_1^2 + \tfrac{1}{2} (t_1 + \eps P)\, \delta \Big) + 2 \eps^3 L_3 \Phi_1.
\]
Then, for every $\theta$,
\[
\Big| w_{AC}(\theta) - \eps F_0^{(1)}(\theta) - \eps^3 \big( w_3(\theta) + F_0^{(3)}(\theta) \big) \Big| \leq R(\eps),
\]
where
\[
w_3(\theta) = -\frac{1}{2} \sum_E c_E\, \big(\nabla g_1(u_E) \cdot q_0(\theta)\big)\,\big(u_E \cdot q_0(\theta)\big).
\]
Moreover, $R(\eps)/\eps^5$ is a non-decreasing function of $\eps$ which has a finite limit as $\eps \to 0$. In particular, $w_{AC} = \eps F_0^{(1)} + \eps^3 (w_3 + F_0^{(3)}) + O(\eps^5)$, uniformly in $\theta$.
\end{prop}
\begin{proof}
Fix $\theta$, let $q = q_{AC}(\theta)$, so that $|q| \leq t_0 \leq \frac{1}{2}$ by Proposition \ref{prop-hover}, let $\Phi_E = \Phi_{u_E}(q)$ and $x_E = u_E - \Phi_E \in \bar{\mathbb{D}}$. By the first bound of Lemma \ref{lem-phi} (and since $f$ is increasing), $|\Phi_u(q')| \leq f(t^2) t^2$ if $|u| = 1$ and $|q'| \leq t$.

\emph{Tilt.} Since $x_E = u_E$ when the tilt is $0$, $Q_\theta(0) = \eps q_0(\theta) + \eps^3 q_3(\theta)$ (notation of the proof of Proposition \ref{prop-hover}). The right-hand sides of the equations which define $Q_\theta(q)$ and $Q_\theta(0)$ differ by at most $2L \max_E |\Phi_E| / d_{\min}$, so, by Lemma \ref{lem-diag},
\[
|q - \eps q_0 - \eps^3 q_3| = |Q_\theta(q) - Q_\theta(0)| \leq \frac{2KL}{d_{\min}} \max_E |\Phi_E|.
\]
With $|\Phi_E| \leq f(t_0^2) t_0^2$, this gives $|q| \leq t_1$. Hence $|\Phi_E| \leq \Phi_1$ and, using the same inequality again, $|q - \eps q_0| \leq \delta$.

\emph{Feet.} Since $\Phi_E = f(|q|^2)\,(u_E\cdot q)\, q$,
\[
\Phi_E - \frac{\eps^2}{2}(u_E \cdot q_0)\, q_0 = \big(f(|q|^2) - \tfrac{1}{2}\big)(u_E\cdot q)\, q + \frac{1}{2}\Big( (u_E \cdot q)(q - \eps q_0) + \big(u_E \cdot (q - \eps q_0)\big)\, \eps q_0 \Big),
\]
and $f(s) - \frac{1}{2} \geq 0$ is increasing, so the norm of this vector is at most $(f(t_1^2) - \frac{1}{2}) t_1^2 + \frac{1}{2}(t_1 + \eps P)\delta$.

\emph{Heights.} By Taylor's formula on the segment $[u_E, x_E] \subset \bar{\mathbb{D}}$, $g_1(x_E) = g_1(u_E) - \nabla g_1(u_E)\cdot \Phi_E + \rho_E$ with $|\rho_E| \leq \frac{1}{2} M_1 |\Phi_E|^2$, and $|g_3(x_E) - g_3(u_E)| \leq L_3 |\Phi_E|$. By Proposition \ref{prop-w}, $w_{AC}(\theta) = \sum_E c_E \big( \eps g_1(x_E) + \eps^3 g_3(x_E) \big)$. Since $\sum_E c_E g_1(u_E) = F_0^{(1)}(\theta)$, $\sum_E c_E g_3(u_E) = F_0^{(3)}(\theta)$ and $\eps^3 w_3(\theta) = -\eps \sum_E c_E \nabla g_1(u_E) \cdot \frac{\eps^2}{2}(u_E\cdot q_0) q_0$, we get
\begin{multline*}
w_{AC} - \eps F_0^{(1)} - \eps^3 (w_3 + F_0^{(3)})
= - \eps \sum_E c_E \nabla g_1(u_E) \cdot \Big(\Phi_E - \frac{\eps^2}{2}(u_E\cdot q_0) q_0\Big) \\
+ \eps \sum_E c_E \rho_E + \eps^3 \sum_E c_E \big(g_3(x_E) - g_3(u_E)\big),
\end{multline*}
and the bound follows from $\sum_E |c_E| = 2$ and the above estimates.

\emph{Monotonicity.} The functions $f$ and $s \mapsto (f(s) - \frac{1}{2})/s = \frac{1}{2}(1+\sqrt{1-s})^{-2}$ are non-decreasing, and so is $L/\eps$. Hence $t_0/\eps$, $t_1/\eps$, $\Phi_1/\eps^2$, $\delta/\eps^3$ and $(f(t_1^2) - \frac{1}{2})t_1^2/\eps^4$ are non-decreasing functions of $\eps$ with finite limits as $\eps \to 0$, and so is $R(\eps)/\eps^5$.
\end{proof}

(The formula for $w_3$ was also checked numerically against a direct computation of $w_{AC}$.) In $w_3$, only the values of $g_1$ and of its radial derivative on the circle play a role. When $F_0^{(1)} \not \equiv 0$, Corollary \ref{cor-small} applies. So the question is: if $F_0^{(1)} \equiv 0$, does $W_3 := w_3 + F_0^{(3)}$ take both signs? The simplest case is $g_3 = 0$ (then $W_3 = w_3$); a non-zero $g_3$ can be used to cancel some Fourier modes of $w_3$, see \S{}\ref{ss-wobble}.

\subsection{Degenerate Fourier modes: Tables vs Terrains}\label{ss-degen}

Let $\gamma(\varphi) = g_1(r\cos\varphi, r\sin \varphi)$ and $\hat \gamma(k)$ be its Fourier coefficients. As in Remark \ref{rmk-fourier}, $F^{(1)}_0(\theta) = - \sum_k m(k) \hat\gamma(k) e^{ik\theta}$ where $m(k) = - \sum_E c_E e^{ik\theta_E}$. Let
\[
\mathcal{K} = \{ k \in \zz \mid m(k) = 0 \}
\]
be the set of \emph{degenerate modes}. Then $F^{(1)}_0 \equiv 0$ if and only if $\hat \gamma$ is supported in $\mathcal{K}$.

\begin{lem}\label{lem-modes}
\begin{enumerate}\renewcommand{\theenumi}{\bfseries (\alph{enumi})}
 \item $0, \pm 1 \in \mathcal{K}$, and $\mathcal{K} = -\mathcal{K}$.
 \item For a rectangle, all odd integers are in $\mathcal{K}$.
 \item For an isosceles trapezoid with feet at the angles $\theta_A = \alpha$, $\theta_B = -\alpha$, $\theta_C = -\beta$ and $\theta_D = \beta$ (where $0 < \alpha < \beta < \pi$), $\tau = \mu = \sin\beta/(\sin \alpha + \sin \beta)$ and
 \[
 m(k) = \frac{2i\big( \sin\alpha \sin(k\beta) - \sin \beta \sin (k \alpha)\big)}{\sin\alpha + \sin \beta},
 \]
 so $k \in \mathcal{K}$ if and only if $U_{k-1}(\cos \alpha) = U_{k-1}(\cos \beta)$,
 where $U_{k-1}$ is the Čebyšëv polynomial of the second kind. For $k=4$ this reads $\cos^2 \alpha + \cos \alpha \cos \beta + \cos^2 \beta = \frac{1}{2}$.
 \item If the table is \emph{rational}, \ie all $\theta_E$ are in $\frac{2\pi}{n}\zz$ for some $n \geq 4$, then $\mathcal{K} \supset \mathcal{K}_n := \{k \in \zz \mid k \equiv 0, \pm 1 \bmod n\}$.
 \item Let $e_E = e^{i\theta_E}$. Then $m(2) \neq 0$, and $k \in \mathcal{K}$ if and only if there are $\alpha_0, \alpha_1, \alpha_2 \in \mathbb{C}$ such that
 \[
 e_E^k = \alpha_0 + \alpha_1 e_E + \alpha_2 \bar e_E \qquad \text{for } E \in \{A,B,C,D\},
 \]
 \ie (for $k \geq 0$) if and only if the remainder of the division of $z^{k+1}$ by $\prod_E (z - e_E)$ has degree at most $2$. In this case, $\alpha_1 = m(k+1)/m(2)$, $\alpha_2 = m(k-1)/\overline{m(2)}$ and $\mathrm{Re}\big(m(2)\, m(k-1)\, \overline{m(k+1)}\big) = 0$.
 \item Let $s_1, \ldots, s_4$ be the elementary symmetric functions of $e_A, e_B, e_C, e_D$. Then $3 \in \mathcal{K}$ if and only if $s_1 = 0$, if and only if the table is a rectangle; $4 \in \mathcal{K}$ if and only if $s_2 = s_1^2$, and then the table is an isosceles trapezoid (or a rectangle); $5 \in \mathcal{K}$ if and only if $s_3 = 2 s_1 s_2 - s_1^3$, and then, up to rotations and reflections, the table is an isosceles trapezoid (or a rectangle), or its feet are at the four zeros of
 \[
 G_\chi(\theta) = 2 \cos(2\theta) - 2\cos(\theta + \chi) + \cos(2\chi)
 \]
 for some $\chi \in \,]0, \frac{\pi}{4}]$. In particular, if $2 \leq k \leq 4$ and $k \in \mathcal{K}$, then the diagonals have the same length.
\end{enumerate}
\end{lem}
\begin{proof}
(a): $\sum_E c_E = 0$, $\sum_E c_E a_E = 0$ and $m(-k) = \overline{m(k)}$. (b): $\theta_C = \pi$, $\theta_D = \theta_B + \pi$ and $\tau = \mu = \frac{1}{2}$, so $m(k) = \frac{1}{2}(1+(-1)^k)(e^{ik\theta_B}-1)$. (c): by symmetry $X$ is on the real axis, which gives $\tau$ and $\mu$; the formula for $m(k)$ is a direct computation, and $\sin(k x)/\sin x = U_{k-1}(\cos x)$. (d): $e^{ik\theta_E}$ only depends on $k \bmod n$, so $m(k) = m(0)$ or $m(\pm 1)$.

(e): Recall that $\sum_E c_E e_E^j = -m(j)$ and $m(-j) = \overline{m(j)}$. If $m(2) = 0$, then $\sum_E c_E\, p(\theta_E) = 0$ for every trigonometric polynomial $p$ of degree at most $2$ (as $m(0) = m(\pm 1) = 0$); this fails for
\[
p(\varphi) = |e^{i\varphi} - e_B|^2\, |e^{i\varphi} - e_D|^2,
\]
for which the sum is $\tau p(\theta_A) + (1-\tau) p(\theta_C) > 0$. The vectors $(1)_E$, $(e_E)_E$ and $(\bar e_E)_E$ of $\mathbb{C}^4$ are linearly independent (since $\bar e_E = 1/e_E$, a relation $\alpha_0 + \alpha_1 e_E + \alpha_2 \bar e_E = 0$ says that the polynomial $\alpha_1 z^2 + \alpha_0 z + \alpha_2$ has four distinct roots) and they are in the kernel of the linear form $x \mapsto \sum_E c_E x_E$ (as $\sum_E c_E = \sum_E c_E e_E = \sum_E c_E \bar e_E = 0$). Hence they span this kernel, which contains $(e_E^k)_E$ if and only if $m(k) = 0$. Multiplying the identity by $e_E$, it says that the $e_E$ are roots of $z^{k+1} - \alpha_1 z^2 - \alpha_0 z - \alpha_2$. Multiplying it by $c_E \bar e_E$, by $c_E e_E$ and by $c_E \bar e_E^k$ and summing over $E$ gives $m(k-1) = \alpha_2 \overline{m(2)}$, $m(k+1) = \alpha_1 m(2)$ and $\alpha_1 \overline{m(k-1)} + \alpha_2 \overline{m(k+1)} = 0$ (using $m(k) = 0$). Substituting the first two identities in the last one and multiplying by $|m(2)|^2$ gives $2\,\mathrm{Re}\big(m(2)\, m(k-1)\, \overline{m(k+1)}\big) = 0$.

(f): Let $Q(z) = \prod_E (z - e_E) = z^4 - s_1 z^3 + s_2 z^2 - s_3 z + s_4$. The coefficients of $z^3$ in the remainders of $z^4$, $z^5$ and $z^6$ modulo $Q$ are $s_1$, $s_1^2 - s_2$ and $s_1^3 - 2 s_1 s_2 + s_3$, so the three equivalences follow from (e). Since $|e_E| = 1$, $\bar s_j = s_{4-j}/s_4$; a rotation of the table by $\omega \in \mathbb{C}$ ($|\omega| = 1$) multiplies $s_j$ by $\omega^j$, so we may assume that $s_4 = 1$, and then $s_3 = \bar s_1$ and $s_2 \in \rr$. The rotation by $i$ and the reflection $z \mapsto \bar z$ preserve this normalisation, and change $s_1$ into $i s_1$ and $\bar s_1$.

If $s_1 = 0$, the centre of the circle is the centroid of the feet; then the feet form two pairs of antipodal points (if $e_A + e_B \neq 0$, the pairs $\{e_A, e_B\}$ and $\{-e_C, -e_D\}$ of unit vectors have the same sum, hence they are equal), which must be the diagonals (two diameters cross), and the table is a rectangle. Conversely, $s_1 = 0$ for rectangles.

If $s_1 \in \rr \cup i\rr$, then, after a rotation by $i$ if necessary, $Q$ has real coefficients, so the set of feet is symmetric with respect to the real axis: the table is an isosceles trapezoid (or a rectangle), or a kite with feet at $1$, $e^{i\beta}$, $-1$, $e^{-i\beta}$ ($0 < \beta < \pi$). For such a kite, $\tau = \frac{1}{2}(1 + \cos \beta)$, $\mu = \frac{1}{2}$ and $m(k) = \cos(k\beta) - \tau - (1-\tau)(-1)^k$, so $m(4) = \cos(4\beta) - 1 \neq 0$ unless $\beta = \frac{\pi}{2}$, and $m(5) = \cos(5\beta) - \cos\beta = 0$ only for $\beta \in \{\frac{\pi}{3}, \frac{\pi}{2}, \frac{2\pi}{3}\}$.

If $4 \in \mathcal{K}$, then $s_1^2 = s_2 \in \rr$, so $s_1 \in \rr \cup i\rr$ and the table is an isosceles trapezoid or a square. If $5 \in \mathcal{K}$ and $s_1 \neq 0$, then $\bar s_1 = s_1(2s_2 - s_1^2)$, \ie $2s_2 = s_1^2 + \bar s_1/s_1$. Writing $s_1 = \rho e^{i\chi}$, the imaginary part of the right-hand side is $(\rho^2 - 1)\sin(2\chi)$, which must vanish. If $\sin(2\chi) = 0$, then $s_1 \in \rr \cup i\rr$, and the table is an isosceles trapezoid, a square, or the kite with $\beta = \frac{\pi}{3}$ (or $\frac{2\pi}{3}$, which is the same up to a reflection). If $\rho = 1$, then $s_2 = \cos(2\chi)$, $Q(z) = z^4 - e^{i\chi} z^3 + \cos(2\chi) z^2 - e^{-i\chi} z + 1$ and $e^{-2i\theta} Q(e^{i\theta}) = G_\chi(\theta)$, so the feet are the zeros of $G_\chi$; up to rotations by $i$ and reflections, $\chi \in [0, \frac{\pi}{4}]$, and $\chi \neq 0$. Finally, the kite with $\beta = \frac{\pi}{3}$ is the case $\chi = \frac{\pi}{4}$: $G_{\pi/4}(\theta) = 2\cos(\theta + \frac{\pi}{4})\big(2\sin(\theta + \frac{\pi}{4}) - 1\big)$ vanishes at $-\frac{\pi}{12}$, $\frac{\pi}{4}$, $\frac{7\pi}{12}$ and $\frac{5\pi}{4}$.
\end{proof}

For instance, the table obtained by cutting a regular hexagon along two opposite vertices (feet at $0, \frac{\pi}{3}, \frac{2\pi}{3}, \pi$) has $\tau = \frac{1}{3}$, $\mu = \frac{2}{3}$ and $m(k)$ only depends on $k \bmod 6$, with $m(2), m(3), m(4) \neq 0$; hence $\mathcal{K} = \mathcal{K}_6$. So this table has degenerate terrains other than the planes (for example $\gamma(\varphi) = \cos(5\varphi)$).

\subsection{Rational tables}\label{ss-rational}

The following statement is a (third order) version of the argument of Theorem \ref{leteo} using a finite orbit of rotations, as suggested by Assumptions \ref{ass-rot2} and \ref{ass-rot}. Note that it does not require the diagonals to be equal.

In this subsection, the terrain is $\eps g_1$, \ie $g_3 = 0$ in the notation of \S{}\ref{ss-3ord}, and the notation of \S{}\ref{ss-degen} is used: $\hat\gamma$ are the Fourier coefficients of $g_1$ on the circle, and $\mathcal{K}_n$ is the set of the modes $k \equiv 0, \pm 1 \bmod n$. If $F_0^{(1)} \equiv 0$, then Proposition \ref{prop-w3} gives $w_{AC} = \eps^3 w_3 + O(\eps^5)$: the third order term $w_3$ of Proposition \ref{prop-w3}, which only depends on $g_1$, decides whether the table wobbles. (A term $\eps^3 g_3$ would only add $F_0^{(3)}$ to $w_3$, and the orbit sums of $F_0^{(3)}$ vanish as well, see Remark \ref{rmk-cex-comments}(a).)

\begin{prop}\label{prop-rational}
Assume the table is rational (all $\theta_E \in \frac{2\pi}{n}\zz$) and that $\hat \gamma$ is supported in $\mathcal{K}_n$. Then $F^{(1)}_0 \equiv 0$ and
\[
\sum_{j = 0}^{n-1} w_3\Big( \theta + \frac{2 \pi j}{n}\Big) = 0 \qquad \text{for all } \theta.
\]
Consequently, if $w_3 \not \equiv 0$, then the table can be stabilised by rotating it on the terrain $\eps g_1$, for all $\eps$ small enough.
\end{prop}
\begin{proof}
Fix $\theta$ and let $v_i = r(\cos(\theta + 2\pi i /n), \sin(\theta + 2\pi i /n))$ ($i \in \zz/n\zz$) be the vertices of the rotated regular $n$-gon. Write $k \in \mathcal{K}_n$ as $k = nj + s$ with $s \in \{0,\pm 1\}$. At $\varphi = \theta + 2\pi i/n$, $e^{ik\varphi} = e^{i(k-s)\theta} e^{is\varphi}$, and $e^{is\varphi}$ is the restriction to $v_i$ of the function $1$, $(x_1 + i x_2)/r$ or $(x_1 - ix_2)/r$. Hence there is an affine function $\ell_\theta$ (real, since $\gamma$ is real) with $\gamma(\theta + 2\pi i /n) = \ell_\theta(v_i)$ for all $i$, and $\ell_{\theta + 2\pi/n} = \ell_\theta$.

The feet of the table rotated by $\theta + \frac{2\pi j}{n}$ are at the vertices $v_{i_E + j}$ (where $\theta_E = \frac{2\pi i_E}{n}$). Since the four points $(v_{i_E+j}, g_1(v_{i_E+j}))$ lie on the graph of $\ell_\theta$, $F_0^{(1)}(\theta + \frac{2\pi j}{n}) = 0$ and $q_0(\theta + \frac{2\pi j}{n}) = \nabla \ell_\theta =: p$ for all $j$: all the rotated tables of the orbit have the same idealised tilt. Hence, with $G(v) = (\nabla g_1(v) \cdot p)(v \cdot p)$,
\[
\sum_{j=0}^{n-1} w_3\Big(\theta + \frac{2\pi j}{n}\Big)
= - \frac{1}{2} \sum_{j=0}^{n-1} \sum_E c_E\, G(v_{i_E + j})
= - \frac{1}{2} \Big(\sum_E c_E\Big) \sum_{i \in \zz/n\zz} G(v_i) = 0.
\]
If $w_3(\theta) \neq 0$ for some $\theta$, then this identity shows that $w_3$ takes both signs. By Proposition \ref{prop-w3}, so does $w_{AC} = \eps^3 w_3 + O(\eps^5)$ for $\eps$ small enough; conclude with Lemma \ref{lem-wobble}.
\end{proof}

\begin{cor}\label{cor-hexagon}
Consider the hexagonal table (feet at $0, \frac{\pi}{3}, \frac{2\pi}{3}, \pi$) and let $g_1$ be $C^2$; the terrain is $\eps g_1$, and $w_3$ is the third order term of Proposition \ref{prop-w3} (with $g_3 = 0$). If $F_0^{(1)} \not \equiv 0$, or if $F_0^{(1)} \equiv 0$ and $w_3 \not \equiv 0$, then the table can be stabilised by rotating it on the terrain $\eps g_1$, for all $\eps$ small enough.
\end{cor}
\begin{proof}
Corollary \ref{cor-small} if $F_0^{(1)} \not \equiv 0$. Otherwise $\hat\gamma$ is supported in $\mathcal{K} = \mathcal{K}_6$, and Proposition \ref{prop-rational} applies.
\end{proof}

The case where both $F_0^{(1)}$ and $w_3$ vanish identically requires going to higher order; for the hexagonal table, this is taken care of by Corollary \ref{cor-hexagon-L} below, which does not rely on an expansion. Note also that the identity of Proposition \ref{prop-rational} is special to degenerate terrains: numerically, for the hexagonal table on a generic terrain, $\sum_j w_{AC}(\theta + \frac{2\pi j}{n})$ is not $0$ (it is of order $L^3$, as is $\int w_{AC}$).

\subsection{Tables which wobble at every angle}\label{ss-wobble}

Without rationality, nothing forces $W_3 = w_3 + F_0^{(3)}$ to change sign. This can be exploited to construct tables which wobble at every angle, on terrains with arbitrarily small Lipschitz constants. Fix a degenerate mode $k \geq 2$ and take $g_1 = \mathrm{Re}(z^k)$ on the unit circle, so that $F_0^{(1)} \equiv 0$. If $\hat\gamma_3(j)$ are the Fourier coefficients of $g_3$ on the circle, then $F_0^{(3)}(\theta) = -\sum_j m(j)\hat\gamma_3(j) e^{ij\theta}$: the term $F_0^{(3)}$ can cancel any non-degenerate Fourier mode of $w_3$, but it does not change the degenerate ones. So the sign of $W_3$ is decided by the degenerate modes of $w_3$, which are here its mean and its modes $\pm k$ (see below), and the mean can be made large with the radial derivative of $g_1$ on the circle.

\begin{prop}\label{prop-closed}
Let $k \geq 2$ be a degenerate mode, let $b \in \mathbb{C}$, and let
\[
g_1(z) = \mathrm{Re}(z^k) + \tfrac{1}{2}(|z|^2 - 1)\, \mathrm{Re}(b z^{2k}),
\qquad
\xi_+ = \frac{m(k-1)}{\overline{m(2)}},
\qquad
\xi_- = \frac{\overline{m(k+1)}}{\overline{m(2)}}.
\]
Then $F_0^{(1)} \equiv 0$, $q_0(\theta) = e^{i\theta}\big(\xi_+ e^{ik\theta} + \xi_- e^{-ik\theta}\big)$ and
\[
w_3(\theta) = \kappa(b) + \mathrm{Re}\big(A_k e^{ik\theta}\big) + \sum_{j \in \{2k, 3k, 4k\}} \mathrm{Re}\big(W_j e^{ij\theta}\big),
\]
where $\kappa(b) = \frac{1}{8}\mathrm{Re}(b S)$ with $S = m(2k+2)\bar\xi_+^2 + 2 m(2k)\bar\xi_+ \xi_- + m(2k-2) \xi_-^2$, and
\begin{align*}
A_k &= \tfrac{k}{4}\big( 2 m(k-2)\, \xi_+ \xi_- + \overline{m(k-2)}\, \bar\xi_-^2 \big),
\\
W_{2k} &= \tfrac{b}{4}\big( (|\xi_+|^2 + |\xi_-|^2)\, m(2k) + \xi_+ \xi_-\, m(2k-2) + \bar\xi_+ \bar\xi_-\, m(2k+2) \big),
\\
W_{3k} &= \tfrac{k}{4}\, m(k-2)\, \xi_+^2,
\\
W_{4k} &= \tfrac{b}{8}\big( \xi_+^2 m(2k-2) + 2 \xi_+ \bar\xi_- m(2k) + \bar\xi_-^2 m(2k+2) \big).
\end{align*}
Consequently, if $m(2k)$, $m(3k)$, $m(4k) \neq 0$ and $g_3(z) = \mathrm{Re}\big( \sum_{j \in \{2k,3k,4k\}} a_j z^j \big)$ with $a_j = W_j / m(j)$, then, for all $\theta$,
\[
W_3(\theta) = \kappa(b) + \mathrm{Re}\big(A_k e^{ik\theta}\big) \geq \kappa(b) - |A_k|.
\]
\end{prop}
\begin{proof}
On the unit circle, $g_1(z) = \mathrm{Re}(z^k)$ and $k \in \mathcal{K}$, so $F_0^{(1)} \equiv 0$. Let $\zeta = e^{ik\theta}$, so that $u_E = e^{i\theta} e_E$ and $u_E^k = \zeta e_E^k$. In the notation of Lemma \ref{lem-modes}(e), $\alpha_2 = \xi_+$ and $\bar\alpha_1 = \xi_-$, so
\[
g_1(u_E) = \mathrm{Re}(\zeta \alpha_0) + \mathrm{Re}(\zeta \bar\xi_- e_E) + \mathrm{Re}(\zeta \xi_+ \bar e_E) = \mathrm{Re}(\zeta\alpha_0) + u_E \cdot q,
\qquad
q = e^{i\theta}(\xi_+ \zeta + \xi_- \bar \zeta).
\]
The four points $(u_E, g_1(u_E))$ are thus on the graph of an affine function with gradient $q$, and $q_0(\theta) = q$.

On the unit circle, $\nabla g_1(u) = k \bar u^{k-1} + \mathrm{Re}(b u^{2k})\, u$ in complex notation (the gradient of $\mathrm{Re}(h)$ is $\overline{h'}$ if $h$ is holomorphic). Let $Z_E = \bar e_E \xi_+ + e_E \bar\xi_-$, so that $u_E \cdot q_0 = \mathrm{Re}\big(\bar e_E (\xi_+\zeta + \xi_-\bar\zeta)\big) = \mathrm{Re}(Z_E \zeta)$ and $\nabla g_1(u_E) \cdot q_0 = k\, \mathrm{Re}\big(e_E^{k-1}(\xi_+\zeta^2 + \xi_-)\big) + \mathrm{Re}(b \zeta^2 e_E^{2k})\, \mathrm{Re}(Z_E\zeta)$. Hence
\[
w_3 = -\frac{k}{2} \sum_E c_E\, \mathrm{Re}\big(e_E^{k-1}(\xi_+\zeta^2 + \xi_-)\big)\, \mathrm{Re}(Z_E \zeta) \; - \; \frac{1}{2} \sum_E c_E\, \mathrm{Re}(b\zeta^2 e_E^{2k})\, \mathrm{Re}(Z_E\zeta)^2.
\]
By $\mathrm{Re}(x)\mathrm{Re}(y) = \frac{1}{2}\mathrm{Re}(xy + x\bar y)$, the first sum only contains the powers $\zeta^{\pm 1}$ and $\zeta^{\pm 3}$, and the second one the powers $1$, $\zeta^{\pm 2}$ and $\zeta^{\pm 4}$; for instance, the mean of $w_3$ is $-\frac{1}{8}\mathrm{Re}\big(b \sum_E c_E e_E^{2k} \bar Z_E^2\big)$. Expanding $Z_E$ and using $\sum_E c_E e_E^j = -m(j)$, $m(-j) = \overline{m(j)}$ and $m(k) = 0$ gives the formula for $w_3$. Finally, $F_0^{(3)}(\theta) = \sum_E c_E\, \mathrm{Re}\big(\sum_j a_j u_E^j\big) = -\sum_j \mathrm{Re}\big(m(j) a_j e^{ij\theta}\big)$ cancels the modes $2k$, $3k$ and $4k$ of $w_3$.
\end{proof}

The mean $\kappa(b)$ of $w_3$ only depends on the Fourier mode $2k$ of the radial derivative of $g_1$ on the circle, which is $\mathrm{Re}(b e^{2ik\varphi})$. Indeed, $(u_E \cdot q_0)^2$ only has the modes $0$ and $\pm 2k$, so a radial slope with Fourier mode $j$ can only contribute to the mean if $j \in \{0, \pm 2k\}$, and a constant radial slope $b_0$ contributes $\frac{b_0}{2}\mathrm{Re}\big(\overline{m(2)} \xi_+ \xi_-\big) = \frac{b_0}{2|m(2)|^2}\mathrm{Re}\big(m(2)\, m(k-1)\, \overline{m(k+1)}\big)$, which vanishes by Lemma \ref{lem-modes}(e). So, as in the heuristic of Remark \ref{rmk-fourier}, a constant radial slope does not help; the radial slope has to depend on the angle. If $S \neq 0$, then $\kappa(b) = |b||S|/8$ for $b$ in the direction of $\bar S$, which is larger than $|A_k|$ for $|b| > 8|A_k|/|S|$. If moreover $m(2k), m(3k), m(4k) \neq 0$, Propositions \ref{prop-closed} and \ref{prop-w3} then show that $w_{AC} > 0$ for every $\theta$ and every $\eps$ small enough: the table wobbles at every angle on the terrain $\eps g_1 + \eps^3 g_3$. Here are two explicit examples.

\begin{teo}\label{teo-cex}
Let the table be the isosceles trapezoid of Lemma \ref{lem-modes}(c) with $r = 1$, $\cos\alpha = \frac{3}{5}$ and $\cos\beta = \frac{\sqrt{23}-3}{10}$ (so $\alpha \approx 53.13$\textdegree{}, $\beta \approx 79.65$\textdegree{}, and both diagonals have length $\approx 1.8326$). For $\eps > 0$, let $g_\eps$ be the terrain which is equal to
\[
\eps\Big( \mathrm{Re}(z^4) - 3\big(|z|^2 - 1\big)\, \mathrm{Im}(z^8) \Big) - \eps^3 \Big( 4.970\, \mathrm{Im}(z^8) + 5.138\, \mathrm{Re}(z^{12}) + 4.173\, \mathrm{Im}(z^{16}) \Big)
\]
on the closed unit disc, and extended by $g_\eps(z) = g_\eps(z/|z|)$ for $|z| > 1$. If $0 < \eps \leq 4 \cdot 10^{-3}$, then $g_\eps$ is $L$-Lipschitz with $L \leq 10\eps + 169\eps^3 \leq 0.0401$, \eqref{eq-smallL} holds, and
\[
w_{AC}(\theta) \geq 0.5\, \eps^3 > 0 \qquad \text{for every } \theta.
\]
Hence, whatever the angle of rotation, the four feet of the table cannot be on the ground at the same time: the table wobbles at every angle.
\end{teo}
\begin{proof}
By Lemma \ref{lem-modes}(c), $4 \in \mathcal{K}$. We apply Proposition \ref{prop-closed} with $k = 4$ and $b = 6i$, for which $\frac{1}{2}(|z|^2 - 1)\mathrm{Re}(b z^8) = -3(|z|^2-1)\mathrm{Im}(z^8)$. The numbers below were computed with interval arithmetic (script \texttt{numerics/always\_wobbles.py} of the repository) and are rounded here. By symmetry, all $m(j)$ are purely imaginary: $m(2) = -0.74196\, i$, $m(3) = -1.15684\, i$ and $m(5) = 1.65543\, i$, so $\xi_+ = -1.55917$ and $\xi_- = -2.23117$; then $|S| = 3.86995$, $\kappa(6i) = 2.90246$ and $A_4 = -1.46864\, i$. The coefficients $a_8 = 4.97037\, i$, $a_{12} = -5.13789$ and $a_{16} = 4.17287\, i$ of Proposition \ref{prop-closed} (with $|m(8)| = 1.88984$, $|m(12)| = 0.35106$, $|m(16)| = 1.06811$) are rounded to three decimals in $g_3$; since the idealised wobbling function of $\mathrm{Re}((\tilde a_j - a_j) z^j)$ is bounded by $|m(j)|\,|\tilde a_j - a_j|$, this changes $W_3$ by less than $10^{-3}$, and
\[
W_3(\theta) \geq 2.9024 - 1.4687 - 0.001 \geq 1.43 \qquad \text{for all } \theta.
\]
(As a check, the script also computes the Fourier coefficients of $W_3$ from the definitions, by a discrete Fourier transform of $64$ values of $W_3$, which is exact since $W_3$ is a trigonometric polynomial of degree $16$.)

Next, the constants of Proposition \ref{prop-w3}. Let $\rho = |z| \leq 1$. Then $|\nabla \mathrm{Re}(z^4)| = 4\rho^3$ and $\|D^2 \mathrm{Re}(z^4)\| = 12\rho^2$. For $v = \mathrm{Re}(b z^8)$ and $h = \frac{1}{2}(|z|^2 - 1) v$, $\nabla h = v z + \frac{1}{2}(|z|^2 - 1)\nabla v$ and $D^2 h = v I + z \otimes \nabla v + \nabla v \otimes z + \frac{1}{2}(|z|^2 - 1) D^2 v$, with $|\nabla v| = 8|b|\rho^7$ and $\|D^2 v\| = 56|b|\rho^6$, so that $|\nabla h| \leq |b|(4\rho^7 - 3\rho^9) \leq |b|$ and $\|D^2 h\| \leq |b|(28\rho^6 - 11\rho^8) \leq 17|b|$. Hence $L_1 \leq 10$ and $M_1 \leq 114$; also $L_3 \leq 8 \cdot 4.970 + 12 \cdot 5.138 + 16 \cdot 4.173 \leq 168.2$, and $|q_0| \leq P := |\xi_+| + |\xi_-| \leq 3.7904$. Solving the linear system which defines $q_3$ in complex notation shows that the idealised tilt of $\mathrm{Re}(a z^j)$ is $e^{i\theta}(a \xi^{(j)}_+ e^{ij\theta} + \bar a \xi^{(j)}_- e^{-ij\theta})$, where $\xi^{(j)}_+ = \frac{D_2 E_1 - D_1 E_2}{2i\, \mathrm{Im}(\bar D_1 D_2)}$ and $\xi^{(j)}_- = \frac{D_2 \bar E_1 - D_1 \bar E_2}{2i\, \mathrm{Im}(\bar D_1 D_2)}$, with $D_1 = e_A - e_C$, $D_2 = e_B - e_D$, $E_1 = e_A^j - e_C^j$ and $E_2 = e_B^j - e_D^j$ (for $j = k$, $\xi^{(k)}_+ = \xi_+$ and $\xi^{(k)}_- = \xi_-$); this gives $|q_3| \leq P_3 := 15.19$. Finally $d_{\min} = 1.83262$ and $K = 4.35905$.

For $\eps = 4\cdot 10^{-3}$, $L \leq 0.0401$, $LK \leq 0.175$ and $\frac{4}{\sqrt{3}}\frac{(LK)^2}{d_{\min}} \leq 0.04$, so \eqref{eq-smallL} holds, and $R(\eps) \leq 0.93\, \eps^3$. Since $R(\eps)/\eps^3$ and $L$ are non-decreasing in $\eps$, the same holds for all $0 < \eps \leq 4\cdot 10^{-3}$, and Proposition \ref{prop-w3} gives $w_{AC}(\theta) \geq \eps^3 (1.43 - 0.93) = 0.5\,\eps^3$. The last claim follows from Lemma \ref{lem-wobble}.
\end{proof}

By Lemma \ref{lem-modes}(f), a table with diagonals of different lengths needs a degenerate mode $k \geq 5$, and the tables with $5 \in \mathcal{K}$ which are not isosceles trapezoids are those whose feet are the zeros of $G_\chi$. The same construction works for them, although with a smaller $|S|$.

\begin{prop}\label{prop-cex-uneq}
Let the table have its feet at the four zeros $\theta_A \approx -21.21$\textdegree{}, $\theta_B \approx 44.51$\textdegree{}, $\theta_C \approx 101.79$\textdegree{}, $\theta_D \approx 234.90$\textdegree{} of $G_{\pi/6}(\theta) = 2\cos(2\theta) - 2\cos(\theta + \frac{\pi}{6}) + \frac{1}{2}$, with $r = 1$; its diagonals have lengths $\approx 1.7577$ and $\approx 1.9918$. For $\eps > 0$, let $g_\eps$ be the terrain which is equal to
\begin{multline*}
\eps \Big( \mathrm{Re}(z^5) - 40\big(|z|^2 - 1\big)\, \mathrm{Re}\big(e^{i\pi/6} z^{10}\big) \Big) \\
+ \eps^3\, \mathrm{Re}\Big( (-12.990 + 5 i)\, z^{10} + (0.765 + 0.379 i)\, z^{15} + (4.046 - 7.952 i)\, z^{20} \Big)
\end{multline*}
on the closed unit disc, and extended by $g_\eps(z) = g_\eps(z/|z|)$ for $|z| > 1$. If $0 < \eps \leq 10^{-3}$, then $g_\eps$ is $L$-Lipschitz with $L \leq 85\eps + 331\eps^3 \leq 0.0851$, \eqref{eq-smallL} holds, and $w_{AC}(\theta) \geq 0.3\,\eps^3 > 0$ for every $\theta$: the table wobbles at every angle.
\end{prop}
\begin{proof}
The polynomial $Q(z) = z^4 - e^{i\pi/6}z^3 + \frac{1}{2}z^2 - e^{-i\pi/6}z + 1$ satisfies $e^{-2i\theta}Q(e^{i\theta}) = G_{\pi/6}(\theta)$. The script \texttt{numerics/always\_wobbles.py} locates four sign changes of $G_{\pi/6}$ (with interval arithmetic), so its four zeros are the roots of $Q$, and $s_3 = e^{-i\pi/6} = 2 s_1 s_2 - s_1^3$; by Lemma \ref{lem-modes}(f), $5 \in \mathcal{K}$. The rest of the proof is the same as for Theorem \ref{teo-cex}, with $k = 5$ and $b = -80 e^{i\pi/6}$ (the direction of $\bar S$). Reducing $z^6$ modulo $Q$ gives $\alpha_1 = -\frac{i}{2}\sin\frac{2\pi}{3}$ and $\alpha_2 = -i \sin\frac{\pi}{3}$ in Lemma \ref{lem-modes}(e), so $\xi_+ = \alpha_2 = -\frac{\sqrt{3}}{2}\, i$ and $\xi_- = \bar\alpha_1 = \frac{\sqrt{3}}{4}\, i$. Moreover $\tau = 0.50721$, $\mu = 0.73529$, $d_{\min} = 1.75765$, $K = 1.54634$, $|S| = 0.21425$, $\kappa(b) = 2.14249$, $|A_5| = 1.63636$, $a_{10} = -12.99038 + 5.00000\, i$, $a_{15} = 0.76503 + 0.37859\, i$ and $a_{20} = 4.04584 - 7.95171\, i$, so that $W_3 \geq 0.504$ (again after rounding the coefficients of $g_3$); $L_1 \leq 5 + |b| = 85$, $M_1 \leq 20 + 21|b| = 1700$ (for $k = 5$, the bounds on $\nabla h$ and $D^2 h$ in the proof of Theorem \ref{teo-cex} become $|b|(5\rho^9 - 4\rho^{11}) \leq |b|$ and $|b|(45\rho^8 - 24\rho^{10}) \leq 21|b|$), $L_3 \leq 330.5$, $P = |\xi_+| + |\xi_-| \leq 1.2991$, $P_3 \leq 25.85$, and, for $\eps = 10^{-3}$, $R(\eps) \leq 0.198\,\eps^3$. Hence $w_{AC} \geq (0.504 - 0.198)\eps^3 \geq 0.3\,\eps^3$ for $0 < \eps \leq 10^{-3}$.
\end{proof}

\begin{rmk}\label{rmk-cex-comments}
(a) \emph{Which tables.} The construction needs a degenerate mode $k \geq 2$ with $S \neq 0$. If $\mathcal{K} = \{0, \pm 1\}$ (the generic case), then $F_0^{(1)} \equiv 0$ forces $g_1$ to be affine on the circle, $q_0$ is constant, and $W_3$ is the idealised wobbling function of some function on the circle (namely $-\frac{1}{2}(\beta - u \cdot q_0)(u\cdot q_0)^2 + g_3$, where $\beta$ is the radial derivative of $g_1$), so it has mean $0$ and changes sign unless $W_3 \equiv 0$. For rational tables and $\hat\gamma$ supported in $\mathcal{K}_n$, the orbit sums of $W_3$ vanish (the proof of Proposition \ref{prop-rational} applies to $F_0^{(3)}$ as well), and again $W_3$ changes sign unless $W_3 \equiv 0$. Also, for the degenerate rational trapezoids of Remark \ref{rmk-orbit-num}(b) and their extra degenerate mode $k = \frac{n}{2}$, $e_E^k = \pm 1$, so that $m(k-1) = \overline{m(k+1)}$, $\xi_+ = \xi_-$ and $S = 2\,\mathrm{Re}\big(\overline{m(2)} \xi_+^2\big) = 2\,\mathrm{Re}\big(m(2) m(k-1) \overline{m(k+1)}\big)/|m(2)|^2 = 0$ by Lemma \ref{lem-modes}(e): the construction does not apply. So, at this order, the construction requires a non-rational table with a degenerate mode $k \geq 2$; by Lemma \ref{lem-modes}(f), $k \geq 5$ if the diagonals do not have the same length.

(b) \emph{Constant $W_3$.} The mode $k$ of $W_3$ can be removed as well, by adding a Fourier mode $k$ to the radial derivative of $g_1$ (with a coefficient obtained by solving a $2\times 2$ real linear system): then $W_3$ is constant, \ie to leading order the table wobbles by the same amount at every angle. For the trapezoid of Theorem \ref{teo-cex}, replacing $\mathrm{Re}(z^4)$ by $\big(1 + 1.0861(|z|^2 - 1)\big) \mathrm{Re}(z^4)$ in $g_1$ (and changing $g_3$ accordingly) gives $W_3 \approx 2.9025$.

(c) \emph{Numerics.} The bounds of Theorem \ref{teo-cex} and Proposition \ref{prop-cex-uneq} are far from sharp. A direct computation of $w_{AC}$ (fixed point iteration of Proposition \ref{prop-hover} on $8192$ angles) gives $\min_\theta w_{AC}/\eps^3 \approx 1.4328$, $1.4324$, $1.4294$ and $1.4189$ for $\eps = 0.002$, $0.004$, $0.01$ and $0.02$ (Theorem \ref{teo-cex}, where $\min W_3 = 1.4329$); for $\eps = 0.02$, the Lipschitz constant of $g_\eps$ is about $0.19$ (slopes up to about $11$\textdegree{}), and \eqref{eq-smallL} does not hold any more, but the fixed point iteration still converges. For the table of Proposition \ref{prop-cex-uneq}, $\min_\theta w_{AC}/\eps^3 \approx 0.507$, $0.510$ and $0.518$ for $\eps = 0.001$, $0.002$ and $0.004$.

(d) \emph{Without $g_3$.} A radial slope alone (\ie $g_3 = 0$) can also make $w_3$ positive, but then all the modes of $w_3$ have to be controlled by the radial slope; for the trapezoid of Theorem \ref{teo-cex}, a linear programme found such a radial slope, given by a trigonometric polynomial of degree $64$ (and none of degree at most $16$).

(e) \emph{The family $G_\chi$.} For the tables of Lemma \ref{lem-modes}(f) with $5 \in \mathcal{K}$, $|S|/8$ is at most about $0.028$ (near $\chi = 0.5$), compared to $0.484$ for the trapezoid of Theorem \ref{teo-cex}; this is why $|b|$ is larger in Proposition \ref{prop-cex-uneq}. As $\chi \to 0$, the feet tend to four vertices of a regular pentagon, and for $\chi = \frac{\pi}{4}$ they are four vertices of a regular hexagon (the kite of Lemma \ref{lem-modes}(f)); both are rational tables, and $S \to 0$ there. The tables of the family are not isosceles trapezoids (after normalising $s_4 = 1$, $s_1 = e^{i\chi} \notin \rr \cup i\rr$), so their diagonals have different lengths; in the limit $\chi \to 0$, they become equal.

(f) \emph{Moving the axis, and a real table.} The wobbling of Theorem \ref{teo-cex} is of order $\eps^3$, whereas moving the axis of rotation by $\delta$ changes the terrain seen by the feet by a term of order $\eps\delta$ in $F_0^{(1)}$, which does not vanish any more (Proposition \ref{prop-axis}). Numerically (moving the axis in $8$ directions, script \texttt{numerics/physical.py}), the table of Theorem \ref{teo-cex} stops wobbling at every angle as soon as $\delta \gtrsim 0.15\, \eps^2 r$: for $\eps = 0.004$, $0.02$, $0.06$ and $0.08$, already for $\delta \approx 3 \cdot 10^{-6} r$, $6\cdot 10^{-5} r$, $6 \cdot 10^{-4} r$ and $10^{-3} r$. The physical size of the effect is also small. All the lengths are proportional to $r$; they are given here for a table of radius $r = 10$\,cm, which could be 3D-printed (for the table of radius $1$\,m of the introduction, multiply them by $10$). The minimal wobble $\min_\theta w_{AC}$ is $r \cdot 1.43\,\eps^3$, \ie about $9$\,nm for $\eps = 0.004$ (where Theorem \ref{teo-cex} applies) and $1.1\,\mu$m for $\eps = 0.02$. For larger $\eps$, outside the range of the theorems (but where the fixed point iteration still converges), the terrain of Theorem \ref{teo-cex} still makes the table wobble at every angle up to $\eps \approx 0.09$, with a minimal wobble of at most about $33\,\mu$m (for $\eps = 0.08$; the terrain then has slopes up to about $37$\textdegree{} and a height of about $16$\,mm). Re-optimising the coefficients of $g_1$ and $g_3$ for a prescribed maximal slope (script \texttt{numerics/physical.py --optimise}) gives at best about $3\,\mu$m, $13\,\mu$m and $31\,\mu$m for maximal slopes of $11$\textdegree{}, $19$\textdegree{} and $27$\textdegree{} (terrains of height $5$ to $10$\,mm). This is below the accuracy of a common 3D printer (around $0.1$\,mm), and the axis would have to be kept fixed to within about $0.15\,\eps^2 r$, \ie about $40\,\mu$m for $\eps = 0.05$. So a 3D-printed always wobbling table seems unrealistic: these tables are mathematical objects, and in practice a slight move of the axis lets you balance the table, in agreement with Martin's experiments on the terrace of the CERN cafeteria \cite[\S{}1]{Mar}.

(g) \emph{Local balancing.} Baritompa, L\"owen, Polster \& Ross \cite[\S{}2 and \S{}4]{BLPR} say that a table can be \emph{balanced locally} if it can be placed with its four feet on the ground and the centre of its circle on the vertical axis, in any position (tilted, upside down or vertical). Rectangles can always be balanced locally on continuous terrains \cite[Theorem 1]{BLPR}, and they do not know whether any other cyclic table can. The trapezoid of Theorem \ref{teo-cex} cannot: for $0 < \eps \leq 4 \cdot 10^{-3}$, it cannot be balanced locally on the terrain $g_\eps$. Indeed, a position which neither turns the table upside down nor makes it vertical is a position $(\theta, q, c)$ (\S{}\ref{s-transv}), and these positions are excluded by Theorem \ref{teo-cex} and Lemma \ref{lem-wobble}. Turning the table upside down acts on its frame as the symmetry $z \mapsto \bar z$, which maps the trapezoid onto itself ($\bar e_A = e_B$ and $\bar e_C = e_D$), so these positions are the previous ones. In a vertical position, let $w \in \mathbb{C}$, $|w| = 1$, be the upward vertical direction in the frame of the table: the foot $E$ is at the height $c + \mathrm{Re}(\bar w e_E)$ and, for two feet on the ground, $|\mathrm{Re}(\bar w (e_E - e_{E'}))| \leq L |e_E - e_{E'}|$, since the horizontal distance between them is at most $|e_E - e_{E'}|$. For $A$ and $B$, $e_A - e_B = \frac{8}{5}\, i$ gives $|\mathrm{Im}\, w| \leq L$; for $A$ and $D$, $e_A - e_D \approx 0.4204 - 0.1837\, i$ then gives $0.4204\, |\mathrm{Re}\, w| \leq 0.4588\, L + 0.1837\, L$, \ie $|\mathrm{Re}\, w| \leq 1.53\, L$, which contradicts $|w| = 1$, as $L \leq 0.0401$.
\end{rmk}

\subsection{The orbit argument without expansion}\label{ss-orbit}

The following remark explains why the naive chain of touchdowns of Remark \ref{rmk-naive-chain} cannot work, and what can be saved from it.

\begin{rmk}[Chains of touchdowns]\label{rmk-chain}
Assume the diagonals have the same length, so that there is a rotation $R_\psi$ of the circle mapping the chord $AC$ onto the chord $BD$ (with $R_\psi a_A = a_D$ and $R_\psi a_C = a_B$, say). A natural idea to recover (H3), the chain of Remark \ref{rmk-naive-chain}, is to define all touchdowns along an orbit at once: put the table down at $\theta_0$, then at $\theta_1 = \theta_0 + \psi$ the feet $A$ and $C$ should land where $B$ and $D$ landed at $\theta_0$, and so on, closing up after finitely many steps when $\psi$ is rational.

However, in the set-up of \S{}\ref{s-transv}, the map $(\theta,q,c) \mapsto (\theta + \psi, q, c)$ (same tilt and height) moves the rod $AC$ exactly onto the position of the rod $BD$: indeed $S_q(R_{\theta+\psi} a_A,0) = S_q(R_\theta a_D,0)$ and $S_q(R_{\theta+\psi} a_C,0) = S_q(R_\theta a_B,0)$. Hence ``$A$ and $C$ land at $\theta_1$ exactly where $B$ and $D$ landed at $\theta_0$'' means that $B$ and $D$ were on the ground at $\theta_0$, \ie that the table was already stable at $\theta_0$: the exact chain presupposes the conclusion. If, on the other hand, the table is put down at $\theta_1$ (with $A$ and $C$ on the ground), its tilt changes, and the landing points move horizontally by $\Phi_u(q) = O(L^2)$; this is exactly the defect of (H3) measured by Proposition \ref{prop-error}.

What can be saved is the following: instead of chaining the touchdowns, one compares all the rotated tables of a finite orbit with one \emph{common} tilt. For this common tilt, all the rotated tables see the same heights at the vertices of the orbit, so the argument of \S{}\ref{s-ideal} applies exactly; the error is then controlled by the discrepancies between the actual tilts and the common one, and these discrepancies are in turn controlled by the wobbling, provided the table has no ``extra'' degenerate residue. This is Theorem \ref{teo-orbit} below. It uses the finite orbit of rotations by $\frac{2\pi}{n}$ (Assumptions \ref{ass-rot2} and \ref{ass-rot}), but not the fact that the diagonals have the same length.
\end{rmk}

Let the table be rational, \ie $\theta_E = \frac{2\pi i_E}{n}$ with $i_E \in \zz$, where $n \geq 4$ is minimal.
Recall that $m(k) = -\sum_E c_E e^{ik\theta_E} = \mu e^{ik\theta_B} + (1-\mu) e^{ik\theta_D} - \tau - (1-\tau) e^{ik\theta_C}$ is the multiplier of the Fourier mode $k$ in the idealised wobbling function (Remark \ref{rmk-fourier} and \S{}\ref{ss-degen}). Then $m(k)$ only depends on $k \bmod n$, and $m(s) = 0$ for $s \equiv 0, \pm 1 \bmod n$ (Lemma \ref{lem-modes}). As in \S{}\ref{ss-intro-setup}, the table is \emph{non-degenerate} if $m(s) \neq 0$ for all other residues $s \in \zz/n\zz$ (this is a property of the table alone), and then we put
\[
\sigma = \min \{ |m(s)| \,:\, s \in \zz/n\zz,\ s \not\equiv 0, \pm 1 \} > 0.
\]
For instance, the hexagonal table is non-degenerate with $\sigma = \frac{2}{\sqrt{3}}$ ($|m(2)| = |m(4)| = \frac{2}{\sqrt 3}$, $|m(3)| = 2$), while rectangles are not (all odd residues are degenerate).

\begin{teo}\label{teo-orbit}
Let the table be rational and non-degenerate, and let $K$ and $d_{\min}$ be as in \S{}\ref{s-transv}. Put
\[
t = \sqrt{2} K L,
\qquad
\lambda' = \frac{4}{\sqrt{3}} \frac{r}{d_{\min}} K L t,
\qquad
\eta = \frac{4}{\sqrt 3} r L t,
\qquad
A = \frac{2K\sqrt{n}}{\sigma\, d_{\min} (1-\lambda')}.
\]
If $t \leq \frac{1}{2}$, $\lambda' < 1$ and $(n+1) A \eta < 1$, then there is an angle at which the table stands on its four feet. More precisely, for every $\theta$, the $n$ numbers $w_{AC}(\theta + \frac{2\pi j}{n})$ ($j \in \zz/n\zz$) are not all of the same strict sign.
\end{teo}
\begin{proof}
The conditions imply \eqref{eq-smallL}, since $LK \leq t/\sqrt{2}$ and $\frac{4}{\sqrt3} \frac{r}{d_{\min}} (LK)^2 = \lambda'/\sqrt{2}$. Fix $\theta$, let $v_i = R_\theta\big(r\cos\frac{2\pi i}{n}, r\sin \frac{2\pi i}{n}\big)$ for $i \in \zz/n\zz$ (the vertices of the rotated $n$-gon), $\theta_j = \theta + \frac{2 \pi j}{n}$, and let $q_j = q_{AC}(\theta_j)$ and $w_j = w_{AC}(\theta_j)$. In the table rotated by $\theta_j$, the foot $E$ is at $R_{\theta_j} a_E = v_{i_E + j}$, so that (Lemma \ref{lem-coord} and Proposition \ref{prop-w})
\[
w_j = \sum_E c_E\, g\big(v_{i_E + j} - \Phi_{v_{i_E+j}}(q_j)\big).
\]

\emph{Step 1: a common tilt.} For $|q| \leq t$, let $G(q) \in \rr^{n}$ be the vector of heights $G(q)_i = g(v_i - \Phi_{v_i}(q))$ and let $p(q) = \frac{2}{n r^2} \sum_i G(q)_i\, v_i$ be the gradient of the least-squares affine fit of $G(q)$ on the vertices (recall that $\sum_i v_i = 0$ and $\sum_i v_i v_i^T = \frac{nr^2}{2}\mathrm{Id}$). Since $|G(q)_i - g(0)| \leq Lr$, we get $|p(q)| \leq 2L \leq t$ (as $K \geq \sqrt{2}$), and by Lemma \ref{lem-phi}, $p$ is $\eta/r$-Lipschitz on the disc of radius $t$. Since $A \geq \sqrt{2}/r$, $\eta/r < A\eta < 1$ and $p$ has a fixed point $\bar q$. Let $\bar G = G(\bar q)$, $\bar c = \frac{1}{n}\sum_i \bar G_i$ and $N_i = \bar G_i - \bar c - v_i \cdot \bar q$. By construction, $N$ is orthogonal to the affine functions on the vertices, \ie its discrete Fourier coefficients of order $0$ and $\pm 1$ vanish.

\emph{Step 2: wobbling and non-affine part.} Write $w_j = \sum_E c_E \bar G_{i_E + j} + e_j$. By Lemma \ref{lem-phi} and $\sum_E |c_E| = 2$, $|e_j| \leq 2 L \cdot \frac{2}{\sqrt{3}} r t\, |q_j - \bar q| = \eta\, |q_j - \bar q|$. Since $\sum_E c_E = 0$ and $\sum_E c_E v_{i_E+j} = R_{\theta_j} \sum_E c_E a_E = 0$, the affine part of $\bar G$ does not contribute: $\sum_E c_E \bar G_{i_E + j} = (CN)_j$, where $(CN)_j = \sum_E c_E N_{i_E + j}$. The circulant operator $C$ multiplies the discrete Fourier mode $s$ by $-m(s)$, so, $N$ having no component of order $0, \pm 1$, $\|N\|_2 \leq \sigma^{-1} \|CN\|_2$. With $W = \max_j |w_j|$ and $D = \max_j |q_j - \bar q|$, this gives
\[
\|N\|_\infty \leq \|N\|_2 \leq \sigma^{-1} \sqrt{n}\, \max_j |w_j - e_j| \leq \sigma^{-1}\sqrt{n}\,(W + \eta D).
\]

\emph{Step 3: tilts and non-affine part.} As in the proof of Proposition \ref{prop-hover}, $Q_{\theta_j}$ is a $\lambda'$-contraction of the disc of radius $t$, which contains $q_j$ and $\bar q$, so $|q_j - \bar q| \leq |Q_{\theta_j}(\bar q) - \bar q|/(1-\lambda')$. Since $\bar G_i - v_i \cdot \bar q = \bar c + N_i$, the components of the vector $Q_{\theta_j}(\bar q) - \bar q$ along the unit vectors of the two diagonals of the table rotated by $\theta_j$ are $(N_{i_A+j} - N_{i_C+j})/d_1$ and $(N_{i_B+j} - N_{i_D+j})/d_2$ (see the proof of Proposition \ref{prop-hover}), so Lemma \ref{lem-diag} gives $|Q_{\theta_j}(\bar q) - \bar q| \leq 2K \|N\|_\infty / d_{\min}$. Hence $D \leq A (W + \eta D)$, \ie $D \leq AW/(1 - A\eta)$.

\emph{Step 4: the orbit sum.} Each vertex is visited exactly once by each foot along the orbit, so $\sum_j \sum_E c_E \bar G_{i_E + j} = \big(\sum_E c_E\big) \sum_i \bar G_i = 0$. Therefore
\[
\Big| \sum_j w_j \Big| = \Big| \sum_j e_j \Big| \leq n \eta D \leq \frac{n A \eta}{1 - A\eta}\, W < W,
\]
because $(n+1)A\eta < 1$. If all the $w_j$ were $>0$, then $\sum_j w_j \geq \max_j w_j = W$, a contradiction; likewise if they were all $<0$. So $w_{AC}$ takes both signs (or vanishes), hence vanishes somewhere, and one concludes with Lemma \ref{lem-wobble}.
\end{proof}

\begin{cor}\label{cor-hexagon-L}
The hexagonal table (feet at $0, \frac{\pi}{3}, \frac{2\pi}{3}, \pi$, with $r = 1$) can be stabilised by rotating it on every terrain with Lipschitz constant $L \leq 0.065$ (slopes up to about $3.7$\textdegree{}).
\end{cor}
\begin{proof}
Here $n = 6$, $K = 2$ (the diagonals make an angle of $60$\textdegree{}), $d_{\min} = \sqrt{3}$ and $\sigma = \frac{2}{\sqrt{3}}$. Then $t = 2\sqrt{2} L$, $\lambda' = \frac{16\sqrt{2}}{3}L^2$, $\eta = \frac{8\sqrt{2}}{\sqrt{3}} L^2$, $A = \frac{2\sqrt{6}}{1-\lambda'}$ and $(n+1)A\eta = \frac{224 L^2}{1-\lambda'}$, which is $<1$ for $L \leq 0.065$.
\end{proof}

\begin{rmk}\label{rmk-orbit-num}
(a) The constants are far from optimal. Numerically, for the hexagonal table and random smooth terrains, the ratio $|\sum_j w_{AC}(\theta_j)| / \max_j |w_{AC}(\theta_j)|$ is about $0.07\, L^2$ (at most $6\cdot 10^{-3}$ for $L = 0.3$), whereas the proof only gives $\frac{nA\eta}{1-A\eta}$, which is $< 1$ for $L \lesssim 0.066$.

(b) Theorem \ref{teo-orbit} does not use \eqref{ass-len}. Among the rational tables with $n \leq 12$ (up to rotations and reflections, $n$ minimal), the only degenerate tables have equal diagonals: the rectangles, and the isosceles trapezoids with feet at $\{0, \frac{\pi}{4}, \frac{3\pi}{4}, \pi\}$ ($n=8$), $\{0,\frac{\pi}{5}, \frac{3\pi}{5}, \frac{4\pi}{5}\}$ ($n=10$), and $\{0, \frac{\pi}{6}, \frac{\pi}{2}, \frac{2\pi}{3}\}$, $\{0, \frac{\pi}{6}, \frac{5\pi}{6}, \pi\}$, $\{0, \frac{\pi}{3}, \frac{5\pi}{6}, \frac{3\pi}{2}\}$ ($n=12$); for all of them the extra degenerate residue is $\frac{n}{2}$. For rectangles, Theorem \ref{teo-rect} applies. For the degenerate trapezoids the question is open. The construction of \S{}\ref{ss-wobble} does not apply to them (Remark \ref{rmk-cex-comments}(a)); for the first of them (feet at $0, \frac{\pi}{4}, \frac{3\pi}{4}, \pi$, which satisfies all the assumptions listed in the introduction and has $4 \in \mathcal{K}$, like the table of Theorem \ref{teo-cex}), the sum of $w_3$ over the orbit of rotations by $\frac{\pi}{4}$ still vanishes numerically (by a cancellation between the pairs of feet exchanged by the symmetry of the trapezoid).

(c) \emph{Admissible slopes.} For the non-degenerate rational tables, the largest Lipschitz constant $L$ allowed by the conditions of Theorem \ref{teo-orbit} (with $r=1$) is given below, as $L$ and as the corresponding slope $\arctan L$ in degrees, for the best and the worst table with a given minimal $n$ (up to rotations and reflections; the tables are given by the indices $i_E$ of their feet, and ``\# tables'' is the number of non-degenerate tables; script \texttt{numerics/orbit\_slopes.py}).
\begin{center}\small
\begin{tabular}{c|c|c|c}
$n$ & \# tables & best table & worst table \\ \hline
$5$ & $1$ & $\{0,1,2,3\}$: $0.107$ ($6.1$\textdegree{}) & (same) \\
$6$ & $2$ & $\{0,1,2,4\}$: $0.105$ ($6.0$\textdegree{}) & $\{0,1,2,3\}$ (hexagon): $0.066$ ($3.8$\textdegree{}) \\
$7$ & $4$ & $\{0,1,3,5\}$: $0.071$ ($4.1$\textdegree{}) & $\{0,1,2,3\}$: $0.041$ ($2.4$\textdegree{}) \\
$8$ & $5$ & $\{0,1,3,6\}$: $0.067$ ($3.9$\textdegree{}) & $\{0,1,2,3\}$: $0.027$ ($1.5$\textdegree{}) \\
$9$ & $10$ & $\{0,1,3,6\}$: $0.069$ ($4.0$\textdegree{}) & $\{0,1,2,3\}$: $0.018$ ($1.0$\textdegree{}) \\
$10$ & $12$ & $\{0,1,3,7\}$: $0.063$ ($3.6$\textdegree{}) & $\{0,1,2,3\}$: $0.013$ ($0.72$\textdegree{}) \\
$11$ & $20$ & $\{0,1,4,8\}$: $0.052$ ($3.0$\textdegree{}) & $\{0,1,2,3\}$: $0.0091$ ($0.52$\textdegree{}) \\
$12$ & $21$ & $\{0,1,4,9\}$: $0.054$ ($3.1$\textdegree{}) & $\{0,1,2,3\}$: $0.0067$ ($0.38$\textdegree{})
\end{tabular}
\end{center}
The worst table is always the one with four consecutive feet: its diagonals are short and almost parallel ($K$ large, $d_{\min}$ small) and $\sigma$ is small. The admissible slope is not monotone in $n$ for the best tables, but it decreases as the tables get finer: the bound contains $\sqrt{n}/\sigma$, and $\sigma$ tends to $0$ along any sequence of rational tables which approximates a given table with $n \to \infty$. Indeed, if the approximated table is irrational, then by Dirichlet's simultaneous approximation theorem there are integers $k \geq 2$ for which $k\theta_B$, $k\theta_C$ and $k\theta_D$ are all close to $0$ modulo $2\pi$, and then $m(k)$ is close to $m(0) = 0$; the same residues $k$ are non-degenerate for the approximants, with $|m(k)|$ small. For instance, for rational approximants of the trapezoid of Theorem \ref{teo-cex} with $n = 24$, $192$ and $768$, one finds $\sigma \approx 0.11$, $0.005$ and $0.007$, and $L \approx 4\cdot 10^{-3}$, $2\cdot 10^{-4}$ and $7 \cdot 10^{-5}$ ($0.24$\textdegree{}, $0.009$\textdegree{} and $0.004$\textdegree{}). So Theorem \ref{teo-orbit} does not give a uniform Lipschitz constant, and the density argument of Remark \ref{rmk-density} cannot be combined with it.

(d) For the hexagonal table, a numerical search (CMA-ES over a $41$-parameter family of terrains with prescribed Lipschitz constant, maximising $\min_\theta w_{AC}/\max_\theta |w_{AC}|$) did not produce any terrain on which the table wobbles at every angle, also for Lipschitz constants beyond the range of Corollary \ref{cor-hexagon-L}: for $L = 0.1$, $0.2$, $0.3$ and $0.45$ (three runs each, $6000$ evaluations per run), the best value of $\min_\theta w_{AC}/\max_\theta|w_{AC}|$ found was $\approx -0.50$, $-0.50$, $-0.50$ and $-0.49$ respectively, and the orbit ratio of (a) stayed below $0.05$. (For $L = 0.3$ and $0.45$ the conditions of Proposition \ref{prop-hover} are not guaranteed, but the fixed point iteration was used as is.) The value $-\frac{1}{2}$ is explained by the proof: on each orbit, the vector $(w_{AC}(\theta_j))_{j \in \zz/6\zz}$ differs by $O(L^2 \max_j |w_{AC}(\theta_j)|)$ from $CN$, which has no Fourier modes of order $0$ and $\pm 1$; such a vector $x$ satisfies $\sum_j x_j (1 + 2\cos\frac{\pi j}{3}) = 0$, \ie $3x_0 + 2x_1 + 2x_5 - x_3 = 0$, so that $\min_j x_j \leq -\frac{1}{2} \max_j x_j$ (take $x_0$ maximal). The search saturates this bound, far from a counterexample.
\end{rmk}

\begin{rmk}[What can the ergodic theorem still do?]\label{rmk-ergodic}
(a) \emph{To balance tables.} In the idealised model, the ergodic theorem and the intermediate value theorem give the same result (Remark \ref{rmk-ivt}(a)): both only use $\int F_0 = 0$. For the actual table, the ergodic averages of $w_{AC}$ along an irrational rotation converge to $\int w_{AC} = \int (w_{AC} - F_0)$, which is $O(L^3)$ but not $0$ in general (\S{}\ref{ss-remains}). So the ergodic theorem only sees the mean, which is precisely what is lost when passing from $F_0$ to $w_{AC}$. Theorem \ref{teo-orbit} uses more: on a finite orbit, the sum of the main term vanishes \emph{exactly}, and the circulant operator is coercive ($\sigma > 0$). For an infinite orbit (an irrational rotation), the analogue of $\sigma$ is $\inf |m(k)|$ over the non-degenerate modes $k$, which is $0$ by the argument of Remark \ref{rmk-orbit-num}(c) (if $e_E^n$ is close to $1$ for the four feet, then $m(n)$, $m(2n)$ and $m(3n)$ are small, and they do not all vanish unless the table is rational, by a Vandermonde argument): this is a small divisor problem, and averaging does not help with it. Theorem \ref{teo-cex} and Proposition \ref{prop-cex-uneq} show that no argument can work for all irrational tables.

(b) \emph{To weaken the regularity of the terrain.} Birkhoff's theorem holds for integrable functions, so in the idealised model the conclusion ``$F_0$ does not have a constant strict sign almost everywhere'' holds for any integrable terrain. But a position with four feet on the ground needs a zero of the wobbling function, hence continuity, and for the actual table the regularity is needed before any averaging: without a bound on the Lipschitz constant, the equal hovering position need not be unique or continuous (Proposition \ref{prop-hover}), and conditions (H1) and (H2) of \S{}\ref{stouchdown} fail. So the ergodic theorem does not seem to help there either; weakening the Lipschitz assumption would rather require a different way of selecting the touchdown positions.
\end{rmk}

\end{document}